\documentclass[a4, 12pt]{amsart}

\usepackage[all]{xy}
\usepackage{mathrsfs, amsthm, amscd, amsmath, amssymb, graphicx}

\usepackage[dvipdfmx]{hyperref}

\newtheorem{theo}{Theorem}[section]
\newtheorem{prop}[theo]{Proposition}
\newtheorem{lemm}[theo]{Lemma}
\newtheorem{cor}[theo]{Corollary}

\newtheorem{prob}[theo]{Problem}
\newtheorem{conj}[theo]{Conjecture}

\numberwithin{equation}{section}

\theoremstyle{definition}
\newtheorem{defi}[theo]{Definition}
\newtheorem{set}[theo]{Setting}
\newtheorem{ex}[theo]{Example}

\theoremstyle{remark}
\newtheorem{rem}[theo]{Remark}

\newcommand{\supp}[0]{\operatorname{Supp}}
\newcommand{\eend}[0]{\operatorname{End}}
\newcommand{\codim}[0]{\operatorname{codim}}
\newcommand{\nd}[0]{\operatorname{nd}}
\newcommand{\rank}[0]{\operatorname{rank}}
\newcommand{\degree}[0]{\operatorname{deg}}

\newcommand{\pr}{{\rm pr}}
\newcommand{\Sym}{{\rm Sym}}

\newcommand{\verti}{{\rm vert}}
\newcommand{\hor}{{\rm hor}}

\newcommand\sO{{\mathcal O}}

\newcommand{\ddbar}{dd^c}
\newcommand{\cV}{\mathcal{V}}

\newcommand{\dbar}{\overline{\partial}}
\newcommand{\I}[1]{\mathcal{I}(#1)}

\newcommand{\holom}[3]{\ensuremath{#1:#2  \rightarrow #3}}

\newcommand{\fiber }[2]{\ensuremath{#1^{-1} (#2)}}
\newcommand{\ep}{\mathbb \varepsilon}

\begin{document}

\title[Projective klt pairs with nef anti-canonical divisor]
{Projective klt pairs with nef anti-canonical divisor}

\author{Fr\'ed\'eric CAMPANA}
\author{Junyan CAO}
\author{Shin-ichi MATSUMURA}

\address{Universit\'e de Lorraine\\  
  Institut de Math\'ematiques \'Elie Cartan, \\
  B.P. 70239, 54506 Vandoeuvre-les-Nancy  Cedex, France}
\email{frederic.campana@univ-lorraine.fr}

\address{Sorbonne Universit\'{e} - Campus Pierre et Marie Curie,
Institut de Math\'{e}matiques de Jussieu, 4, Place Jussieu,  75252 Paris Cedex 05, France}
\email{{\tt junyan.cao@imj-prg.fr}}

\address{Mathematical Institute, Tohoku University, 
6-3, Aramaki Aza-Aoba, Aoba-ku, Sendai 980-8578, Japan.}
\email{{\tt mshinichi-math@tohoku.ac.jp, mshinichi0@gmail.com}}

\date{\today, version 0.01}

\renewcommand{\subjclassname}{%
\textup{2010} Mathematics Subject Classification}


\maketitle

\begin{abstract}
In this paper, we study a projective klt pair $(X, \Delta)$ 
with the nef anti-log canonical divisor $-(K_X+\Delta)$ and 
its maximally rationally connected fibration $\psi: X \dashrightarrow Y$. 
We prove that the numerical dimension of the anti-log canonical divisor 
$-(K_X+\Delta)$ on $X$ coincides with that of the anti-log canonical divisor 
$-(K_{X_y}+\Delta_{X_y})$ on a general fiber $X_y$  of 
$\psi: X \dashrightarrow Y$, 
which is an analogue of  Ejiri-Gongyo's result formulated for the Kodaira dimension. 
As a corollary, we reveal a relation between positivity of the anti-canonical divisor and the rationally connectedness, 
which gives a sharper estimate than the question posed by Hacon-$\mathrm{M^{c}}$Kernan. 
Moreover, in the case of $X$ being smooth, 
we show that 
a maximally rationally connected fibration $\psi: X \to Y$ 
can be chosen to be a morphism to 
a smooth projective variety $Y$ with numerically trivial canonical divisor, 
and further that it is locally trivial with respect to the pair $(X, \Delta)$, 
which can be seen as a generalization of  Cao-H\"oring's structure theorem
to klt pair cases. Finally, we study the structure of the slope rationally connected quotient 
for a pair $(X, \Delta)$ with $-(K_X +\Delta)$ nef,
and obtain a structure theorem for projective orbifold surfaces.
\end{abstract}

\tableofcontents

\section{Introduction}\label{Sec1}

The study of certain \lq \lq positively curved" varieties, 
which are often formulated with positivity of bisectional curvatures, tangent bundles, or anti-canonical divisors, 
occupies an important place in the theory of classification of varieties. 
One of the central problems in this study is 
to determine the structure of semi-positively curved varieties 
in terms of naturally associated fibrations such as Albanese maps, Iitaka fibrations, or maximally rationally connected fibrations. 
Particularly, the study of compact K\"ahler manifolds with nef anti-canonical divisor, 
which has been rapidly developed in the last decades, 
covers a large range of positively curved varieties, and thus it is an attractive topic.

Cao-H\"oring recently established a structure theorem 
for a smooth projective variety $X$ with nef anti-canonical divisor in \cite{CH}: 
$X$ admits a locally trivial fibration $X \to Y$ 
to a smooth projective variety $Y$ with numerically trivial canonical divisor 
such that the fiber is rationally connected. 
By combining the Beauville-Bogomolov decomposition (see \cite{Bea83}), 
it is also shown that the universal cover  of $X$ can be decomposed 
into the product of $\mathbb{C}^m$, Calabi-Yau manifolds, hyperk\"ahler manifolds, and rationally connected manifolds. 
For various previous researches, 
see \cite{CP91, DPS93, DPS94, DPS96, Zha96, Pau97, CDP15, Cao, CH}, and references therein.

This paper is devoted to the study of the structure of a projective klt pair $(X, \Delta)$ 
with the nef anti-log canonical divisor $-(K_X +\Delta)$. 
We study its maximally rationally connected (\lq \lq MRC" for short) fibration 
and its slope rationally connected (\lq \lq sRC" for short)  quotient  introduced by Campana in \cite{Cam16}.
This paper contains three main results: 
The first main result solves 
the question posed by Hacon-$\mathrm{M^{c}}$Kernan in \cite{HM}  
in a more general form (see Theorem \ref{thm-main}). 
The second main result establishes a generalization of Cao-H\"oring's structure theorem 
to projective klt pairs  (see Theorem \ref{thm-maincam}). 
The third main result gives a structure theorem of sRC quotients 
for a projective surface $(X, \Delta)$ with the nef anti-log canonical divisor $-(K_X +\Delta)$ 
(see Theorem \ref{thm-sRC-intro}).

\subsection{Positivity of anti-canonical divisors --Hacon-$\mathrm{M^{c}}$Kernan's question--}

The following question posed by Hacon-$\mathrm{M^{c}}$Kernan asks a relation 
between the rationally connectedness of $X$ and positivity of $-(K_X +\Delta)$. 
This question can be seen as a generalization of the well-known fact 
that any weak Fano varieties are rationally connected (see \cite{Cam92} and \cite{KoMM92}). 
Indeed, the condition of $ \dim Y=0$ is equivalent to the rationally connectedness of $X$, 
and thus the right hand side in the question measures how far $X$  is from rationally connected varieties. 

\begin{prob}[{Hacon-$\mathrm{M^{c}}$Kernan's question, \cite[Question 3.1]{HM}}]\label{prob-HM}
Let $(X, \Delta)$ be a projective klt pair with the nef anti-log canonical divisor $-(K_X +\Delta)$ 
and let $\psi: X \dashrightarrow Y$ be an MRC fibration of $X$. 
Then does the following inequality hold?
\begin{align*}
\kappa (-(K_X+\Delta)) \leq \dim X - \dim Y,  
\end{align*}
where $\kappa(\bullet)$ is the Kodaira dimension. 
\end{prob}

Ejiri-Gongyo affirmatively solved the above question in \cite{EG} by proving the following inequality:
\begin{align*}
\kappa (-(K_X+\Delta)) \leq \kappa (-(K_{X_y}+\Delta_{{X_y}})), 
\end{align*}
where  $X_y$ is a general fiber of the MRC fibration $\psi: X \dashrightarrow Y$ 
and $\Delta_{_{X_y}}$ is the effective divisor defined by the restriction 
$\Delta_{_{X_y}}:=\Delta |_{X_y}$ to $X_y$. 
In this paper, we prove an analogue of Ejiri-Gongyo's result 
by replacing the Kodaira dimension  with the numerical dimension. 
Moreover, as an application, 
we resolve Hacon-$\mathrm{M^{c}}$Kernan's question in the following generalized form:


\begin{theo}\label{thm-main}
Let $(X, \Delta)$ be a projective klt pair with the nef anti-log canonical divisor $-(K_X +\Delta)$ 
and let $\psi: X \dashrightarrow Y$ be an MRC fibration of $X$. 
Then the following equality  holds\,$:$ 
$$
\nd (-(K_X+\Delta))=\nd (-(K_{X_y}+\Delta_{_{X_y}}))
$$ 
for a general fiber $X_y$ of $\psi$. 
Here $\nd (\bullet)$ is the numerical dimension and  
$-(K_{X_y}+\Delta_{{X_y}})$ is the anti-log canonical divisor on a general fiber $X_y$.  

In particular, we obtain a generalization of Hacon-$M^c$Kernan's question\,$:$  
$$
\nd (-(K_X+\Delta)) \leq \dim X - \dim Y. 
$$ 
\end{theo}

We emphasize that there are several advantages in considering the numerical dimension instead of the Kodaira dimension. 
For example, it follows that the inequality in Theorem \ref{thm-main} gives a sharper inequality 
than Hacon-$\mathrm{M^{c}}$Kernan's question proved by  Ejiri-Gongyo 
from the well-known formula $\kappa (\bullet) \leq  \nd (\bullet)$. 
The interesting point here is that the equality in Theorem \ref{thm-main} does not hold   
without considering the numerical dimension. 
Indeed, there is a counterexample to the equality for the Kodaira dimension 
although we have Ejiri-Gongyo's inequality (see Example \ref{exa}).

The inequality of $\nd (-(K_X+\Delta)) \geq \nd (-(K_{X_y}+\Delta_{_{X_y}}))$ in Theorem \ref{thm-main} 
is proved by an extension theorem of holomorphic sections 
from a general fiber $X_y$ to the ambient space $X$. 
On the other hand, the converse inequality is obtained from 
considering the direct image sheaf of pluri-anti-canonical divisors $-m(K_{X/Y}+\Delta)$. 
The key points in this step are to obtain certain flatness 
and to estimate the dimension of the space of holomorphic sections of the direct image sheaf.

\subsection{Structure theorem of MRC fibrations for klt pairs}

From the viewpoint of birational geometry, 
it is natural to ask whether or not Cao-H\"oring's structure theorem 
can be extended to log pairs. 
This paper answers this question in a quite general form 
(see Theorem \ref{splitgener}), which is one of contributions of this paper. 
Moreover, by considering the case where $X$ being smooth in detail, 
we generalize Cao-H\"oring's structure theorem 
to the klt pairs as follows:

\begin{theo}\label{thm-maincam}
Let $(X, \Delta)$ be a projective klt pair such that  $X$ is smooth and 
the anti-log canonical divisor $-(K_X +\Delta)$ is nef. 
Then there exists an MRC fibration $\psi: X \rightarrow Y$ with the following properties$:$
\begin{enumerate}
\item[(1)] $Y$ is smooth and $c_1 (Y)=0$. 
\item[(2)] $\psi$ is locally trivial with respect to $(X, \Delta)$, 
namely,  for any small open set $U \subset Y$, we have the isomorphism
$$(\psi^{-1} (U), \Delta )\cong  U \times (X_y, \Delta_{X_y})$$
over $U \subset Y$.
Here $X_y$ is a general fiber of $\psi$.
\end{enumerate} 

Moreover, together with the Beauville-Bogomolov decomposition, 
the universal cover $X_{\rm univ}$ of $X$ can be decomposed into 
the product of $\mathbb{C}^m$, Calabi-Yau manifolds, hyperk\"ahler manifolds, and rationally connected manifolds. 
\end{theo}
The proof of Theorem \ref{splitgener} and Theorem \ref{thm-maincam} 
gives a generalization of the argument in \cite{CH} 
to the singularities of pairs, and also it contains several new ingredients 
(for example, Proposition \ref{propnumflat}, the argument of Theorem \ref{thm-maincam}, subsection \ref{subsec-4.3}).  
It is worth to mention that our proof  
provides not only a generalization but also a \lq \lq simplified and unified" version of \cite{Cao} and \cite{CH}. 


\subsection{Structure theorem of sRC quotients for orbifolds}

The last part of this paper is devoted to the study of the structure of 
the slope rationally connected quotient (\lq \lq sRC quotient" for short). 
The sRC quotient, which was established in \cite{Cam16}, is a natural fibration associated  to a pair $(X, D)$. 
The idea is that, in the case with the additional information of pairs $D$, the sRC quotient is more precise than the MRC fibration.
Let us first recall the definition of sRC manifolds, following \cite{Cam16}.
Let $(X, D)$ be a log canonical pair (\lq \lq lc pair" for short)  with $X$ smooth. 
We say that $(X, D)$ is {\textit{sRC}}, if for any ample line bundle $A$ on $X$ and 
some positive integer $m_A$,
we have 
$$
h^0 (X', \pi^\star  (\otimes^m (\Omega^1 (X,D))) \otimes \pi^\star A)=0 \qquad\text{for any } m \geq m_A,
$$
where $\pi: X' \rightarrow (X,D)$ is some (or any) Kawamata cover adapted to $(X, D)$.  
We refer to \cite{Cam16} for a detailed discussion of sRC manifolds. 
The following theorem has been proved in \cite{Cam16}:

\begin{theo}[{\cite[Theorem 1.5, Corollary 10.6]{Cam16}}] \label{thm-cam}
Let $(X, D)$ be an lc pair with $X$ smooth. 
Then there exists an orbifold birational model $(X', D')$ 
and an orbifold morphism $\rho : (X', D')\rightarrow (R, D_R)$ onto its $($smooth$)$ orbifold base $(R, D_R)$ 
with the following two properties:

\begin{enumerate}
\item  The smooth orbifold fibers $(X_r, D_r )$ are sRC.

\item The divisor $K_R + D_R$ is pseudo-effective.
\end{enumerate}

Such a fibration is uniquely determined, up to orbifold birational equivalence.
Moreover, in the case $-(K_X +D)$ is nef, the orbifold base $(R, D_R)$ satisfies that 
$$\kappa (K_R +D_R) = \nd (K_R +D_R)=0.$$
\end{theo}

We remark that an orbifold pair $(X, D)$ might not be sRC in the case where $D$ is a nontrivial divisor 
even if $X$ itself is rational connected. 
Hence the sRC quotient of $(X, D)$ might not be trivial even if the MRC fibration of $X$ is trivial. 
This implies that we can extract more information from the sRC quotients than only MRC fibrations. 
It is natural to ask whether or not we can expect  structure theorems for sRC quotients 
in the same spirit  as Theorem \ref{thm-maincam}. 
In this paper, we pose the following conjecture for sRC quotients. 

\begin{conj}\label{mainconsRC}
Let $(X, D)$ be a klt pair such that $X$ is smooth and $- (K_X +D)$ is nef. 
Then there exists an orbifold morphism $\rho : (X, D)\rightarrow (R, D_R)$ with the following properties$:$
\begin{enumerate}
\item  $(R, D_R)$ is a klt pair such that $R$ is smooth and $c_1(K_R+ D_R)=0$.

\item  General orbifold fibers $(X_r, D_r )$ are sRC.

\item The fibration is locally trivial with respect to pairs, namely, for any small open set $U \subset R$, we have the isomorphism
$$
(\rho^{-1} (U), D)\cong  (U, D_R |_U) \times (X_r, D_{X_r})
$$
over $U \subset Y$. 
Here $X_r$ is a general fiber of $\rho$.

\end{enumerate}
\end{conj}

In this paper, we confirm the conjecture 
in the case of $(X, D)$ being a smooth orbifold surface  
(see Theorem \ref{thm-sRC-intro}), 
and observe the general case (see Theorem \ref{tfsrc}). 
Further we provide Example \ref{exlc}, which says that 
the above structure result can not be expected for (non-klt) lc pairs. 

\begin{theo}\label{thm-sRC-intro}
Let $(X, D)$ be a klt pair such that 
$X$ is a smooth surface, 
the support of $D$ is normal crossing, 
and $- (K_X +D)$ is nef. 
Then Conjecture \ref{mainconsRC} is true.

Moreover, if we assume that $h^1(X,\mathcal{O}_X)=0$ and $(X,D)$ is not sRC, 
then $(X,D)$ is a product with respect to pairs, 
namely, we have the isomorphism
$$
(X,D)\cong(R,D_R)\times (\mathbb{P}^1,D_{\mathbb{P}^1}). 
$$
\end{theo}

When $(X,D)$ is a smooth klt surface pair with $-(K_X+D)$ nef, 
the sRC-fibration gives a more precise description 
than the classical MRC fibration (which is trivial), 
except when $(X,D)$ is sRC, or when $-(K_X+D)$ is numerically trivial.

\smallskip

The organization of this paper is as follows: 
Our methods heavily depend on positivity of direct image sheaves and 
properties of numerical flatness. 
For this reason, 
in Section \ref{Sec2}, 
we develop the theory of positivity of direct image sheaves and 
numerically flat vector bundles. 
In Section \ref{Sec3}, for the proof of Theorem \ref{splitgener}, 
we generalize key propositions obtained in \cite{Cao} and \cite{CH} 
to a morphism with nef relative anti-canonical divisor. 
In Section \ref{Sec4}, we prove  
Theorem \ref{thm-main}, Theorem \ref{splitgener}, and Theorem \ref{thm-maincam}, 
by applying the theory established in Section \ref{Sec2} and Section \ref{Sec3}. 
In Section \ref{sRC}, we discuss some properties of sRC quotients for orbifolds 
with nef anti-log canonical divisors.



\subsection*{Acknowledgements}
The third author would like to thank Prof. Yoshinori Gongyo for discussion on \cite{EG}. 
He also would like to express his thank to 
the members of Institut de Math\'ematiques de Jussieu\,-\,Paris Rive gauche 
for their hospitality during my stay. 
He is supported by the Grant-in-Aid 
for Young Scientists (A) $\sharp$17H04821 from JSPS 
and the JSPS Program for Advancing Strategic International Networks 
to Accelerate the Circulation of Talented Researchers.
This work is partially supported by the Agence Nationale de la Recherche grant 
\lq \lq Convergence de Gromov-Hausdorff en g\'{e}om\'{e}trie k\"{a}hl\'{e}rienne" (ANR-GRACK).

\section{Preliminaries}\label{Sec2}
In this section, after we summarize several results for  positivity of direct image sheaves, 
we develop the theory of numerically flat vector bundles. 
The new ingredients here are Proposition \ref{propnumflat}, 
Proposition \ref{isotri}, and Proposition \ref{prop-rank}. 

\subsection{Positivity of direct image sheaves}\label{Sec2-1}
In this subsection, we shortly recall the notion of singular hermitian metrics on vector bundles 
(more generally torsion free sheaves) and the theory of positivity of direct image sheaves 
(see \cite{Ber09}, \cite{BP08}, \cite{PT18} for more details). 

Let $E$ be a (holomorphic) vector bundle of rank $r$ on a complex manifold $Y$. 
A (possibly) singular hermitian metric $h$ on $E$ is a hermitian metric that 
can be locally written as a measurable map with the values 
in the set of semi-positive (possibly unbounded) hermitian forms on $\mathbb{C}^r$ 
satisfying that $0 < \det h < +\infty$ almost everywhere. 
A singular hermitian vector bundle $(E, h)$ is said to be 
\textit{positively curved}, 
if $\log |s|_{h^{\star}}$ is a plurisubharmonic (psh for short) function on $U$
for any open set $U \subset Y$ and any section $s \in H^{0}(U, E^\star)$, 
where $h^{\star}:={}^th^{-1}$ is the dual hermitian metric of $h$ on the dual bundle $E^{\star}$. 
When $h$ is a smooth  hermitian metric, 
the curvature associated to $h$ is semi-positive in the sense of Griffith 
if and only if $(E,h)$ is positively curved. 
For a smooth $(1,1)$-form $\theta$ on an open set $U \subset Y$, 
we simply write    
$$
\sqrt{-1}\Theta_{h}(E) \succeq \theta \otimes {\rm{Id}}_E \text{ on }  U 
$$
if $\ddbar \log |s|_{h^{\star}}-\theta$ is semi-positive in the sense of currents 
for any open set $U' \subset U$ and any section $s \in H^{0}(U', E^\star)$. 
The notation of $\sqrt{-1}\Theta_{h}(E)$ is used in the above definition, 
but an appropriate definition of the curvature is unknown 
for singular hermitian vector bundles in general. 
We now introduce the notion of weakly positively curved torsion free sheaves.

\begin{defi}\label{def-weak}
Let $E$ be a torsion free (coherent) sheaf on a complex manifold $Y$ 
and $\omega_Y$ be a fixed hermitian form on $Y$. 
The sheaf $E$ is said to be {\textit{weakly positively curved}} 
if for any real number $\ep >0$ there exists a singular hermitian metric $h_\ep$ on 
$E|_{Y_1}$ 
such that 
$$
\sqrt{-1} \Theta_{h_\ep} (E) \succeq -\ep \omega_Y \otimes {\rm{Id}}_E  \text{ on } Y_1, 
$$ 
where $Y_1$ is the maximal locally free locus of $E$. 
\end{defi}

It follows that, 
for a proper morphism  $\varphi: \Gamma \to Y$ with connected fiber, 
the direct image sheaf of the pluri-relative canonical bundle $m K_{\Gamma/Y}$ 
twisted by a line bundle $L$
$$
\varphi_\star(m K_{\Gamma/Y}+L) 
$$
admits a positively curved singular hermitian metric 
under suitable assumptions, 
from the fundamental works developed by several authors 
including Berndtsson, P\u{a}un, Takayama 
(for example see \cite{Ber09}, \cite{BP08}, \cite{PT18}, and references therein). 
The results for positivity of direct image sheaves, 
which play an important role in this paper, 
can be summarized in the following form. 
In the following theorem, 
the first conclusion on positivity of direct image sheaves 
is derived from \cite{BP08}, 
\cite[Lemma 5.25]{CP17}, and \cite{PT18}. 
Also the latter conclusion is an application of the first conclusion. 
Here we give only the proof of the latter conclusion.

\begin{theo}[{\cite[Lemma 5.25]{CP17}, \cite{PT18}, \cite{BP08}}]\label{thm-cur}
Let $\varphi: \Gamma \to Y$ be a morphism with connected fiber between smooth projective varieties 
and $(L, h)$ be a singular hermitian line bundle on $\Gamma$. 
Let $m\in\mathbb{Z}_+$ such that 
$\I{h^{1/m}|_{\Gamma_y}}=\mathcal{O}_{\Gamma_{y}}$ holds 
for a general fiber $\Gamma_y$. 
Here $\I{\bullet}$ denotes the multiplier ideal sheaf of 
a singular hermitian metric $\bullet$. 

$(1)$ If $\sqrt{-1}\Theta_{h}(L) \geq \varphi^\star \theta $ for a smooth $(1,1)$-form $\theta$ on $Y$,  then the induced singular hermitian metric $H$ on the direct image sheaf 
$\varphi_\star(mK_{\Gamma/Y}+L)|_{Y_1}$ 
satisfies that 
$$
\sqrt{-1}\Theta_H (\varphi_\star(mK_{\Gamma/Y}+L))\succeq \theta \otimes {\rm{Id}}_Y  \text{ on } Y_1, 
$$
where $Y_1$ is the maximal locally free locus of $\varphi_\star(mK_{\Gamma/Y}+L)$.

$(2)$ We further assume that 
$L$ is $\varphi$-big and it satisfies that 
$\sqrt{-1}\Theta_{h}(L) \geq 0$ in the sense of currents.
Then, for a nef line bundle $N$ on $\Gamma$,  the direct image sheaf 
$\varphi_\star(mK_{\Gamma/Y}+N+ L)$ is weakly positively curved on $Y$.
\end{theo}
\begin{proof}[{Proof of Theorem \ref{thm-cur} $(2)$}]
There are a singular hermitian metric $H$ on $L$  and 
a smooth hermitian metric $g_{\delta} $ on $N$
such that 
$$
\sqrt{-1}\Theta_H(L) + \varphi^{\star} \omega_Y \geq \omega \text{ and } 
\sqrt{-1}\Theta_{g_{\delta}}(N) \geq - \delta \omega
$$
for some K\"ahler form $\omega_Y$ on $Y$ and some K\"ahler form $\omega$ on $\Gamma$, 
since $L$ is $\varphi$-big and $N$ is nef. 
Then the metric $h_{\ep}$ on $L+N$
defined by $h_{\ep}:=h^{1-\ep} H^{\ep} g_{\ep}$ satisfies that 
$$
\I{h_{\ep}^{1/m}|_{\Gamma_y}}=\I{h^{(1-\ep)/m} H^\ep|_{\Gamma_y}}
=\mathcal{O}_{\Gamma_{y}} \text{ and }
\sqrt{-1}\Theta_{h_{\ep}}(L+N) \geq -\ep \varphi^{\star} \omega_Y 
$$
for any $1 \gg \ep \gg \delta >0$. 
Therefore the conclusion follows from the first conclusion. 
\end{proof}

As a corollary, we have the following lemma, which will be useful for us.

\begin{lemm}\label{withpairlem}
Let $\varphi: X\rightarrow Y$ be a morphism with connected fiber 
between smooth projective varieties 
and let $\Delta$ be a klt divisor on $X$ such that $- (K_{X/Y} +\Delta)$ is nef.
Let $L$ be a pseudo-effective and $\varphi$-relatively big line bundle on $X$.
Then, for any $p, q \in \mathbb{Z}_+$, 
the direct image sheaf $\varphi_\star (p \Delta + q L)$ is weakly positively curved on $Y$.
\end{lemm}

\begin{proof}
Let $h_L$ be a possible singular hermitian metric on $L$ such that $ \Theta_{h_L} (L) \geq 0$ on $X$. Let $h_\Delta$ be the canonical singular metric 
defined by the effective divisor $\Delta$. 
For $m$ large enough with respect to $p, q$, 
it can be seen that  
$$
\mathcal{I} (h_L ^{\frac{q}{m}} \cdot h_\Delta ^{\frac{m+p}{m}}|_{X_y}) =\mathcal{O}_{X_y}
$$ 
for a general fiber $X_y$.
Here we used the assumption of $\Delta$ begin a klt divisor.  
By applying Theorem \ref{thm-cur} $(2)$ to 
$$p\Delta + q L = m K_{X/Y}  - m (K_{X/Y} +\Delta) + (p+m)\Delta + q L,$$
we can see that $\varphi_\star (p\Delta + q L )$ is weakly positively curved on $Y$.
\end{proof}

%

\subsection{Numerically flat vector bundles}\label{Sec2-2}

In this subsection, 
after we recall the notion of numerically flat vector bundles, 
we show that any weakly positively curved reflexive sheaf with numerically trivial first Chern class 
is actually a vector bundle and numerically flat. 
It was proved in \cite[Cor 2.12]{CH} for projective surfaces.
Together with a Bando-Siu's result (see \cite{BS}) 
for (admissible) Hermitian-Einstein metrics, we can generalize \cite[Cor 2.12]{CH} to arbitrary dimension.

We first recall the definition of numerically flat vector bundles and 
the fundamental work of \cite{DPS94} on a characterization 
in terms of filtrations of hermitian flat vector bundles. 

\begin{defi}\label{def-numflat}
Let $E$ be a vector bundle on a compact complex manifold $Y$. 
The vector bundle $E$ is said to be \textit{numerically flat} 
if $E$ is nef and its dual bundle $E^\star$ is also nef, 
equivalently  $E$ is a nef vector bundle satisfying that 
$$
c_{1}(E)=c_1(\det E) =0 \in H^{1,1}(Y, \mathbb{R}).  
$$
Here a vector bundle $E$ is said to be \textit{nef} 
if the tautological line bundle $\mathcal{O}_{\mathbb{P}(E)}(1)$ is nef 
on $\mathbb{P}(E)$. 
\end{defi}


\begin{theo}[\cite{DPS94}]\label{thm-numflat}
Let $E$ be a  vector bundle on a compact K\"ahler manifold $Y$. 
Then the following conditions are equivalent$:$
\begin{itemize}
\item[$\bullet$] $E$ is numerically flat. 
\item[$\bullet$]  There exists a filtration of $E$ by subbundles 
$$
0=:E_0 \subset E_1 \subset \dots \subset E_{p-1} \subset E_p:=E 
$$
such that each quotient bundle $E_k/E_{k-1}$ 
is a hermitian flat vector bundle for any $1 \leq k \leq p$ 
$($that is, the quotient $E_k/E_{k-1}$ admits a smooth hermitian metric $h_k$ such that 
$\sqrt{-1}\Theta_{h_k}(E_k/E_{k-1})=0$$)$. 

\end{itemize}
\end{theo}

The following lemma, which is needed for the proof of Proposition \ref{propnumflat}, 
easily follows from an induction on dimension and 
the standard argument by the restriction to hypersurfaces, 
and thus we omit the proof.

\begin{lemm}\label{lemm-elementary}
Let $E$ be a vector bundle on a smooth projective variety $Y$ and 
$Y_0$ be a Zariski open set in $Y$ with $\codim (Y \setminus Y_0) \geq 2+i$.
Then the morphism induced by the restriction 
$$
H^{j}(Y, E) \to H^{j}(Y_0, E)
$$
is an isomorphism for any $j \leq i$. 
\end{lemm}

The following proposition was proved when $Y$ is a surface in \cite{CH}. 
We now generalize it to arbitrary dimension 
by using the complete intersection argument and Bando-Siu's theorem.
Note that, by the same method as in this paper, 
it is slightly generalized to pseudo-effective reflexive sheaves (see \cite{HIM19}). 

\begin{prop}\label{propnumflat}
Let $\mathcal{F}$ be a reflexive sheaf on a smooth projective variety $Y$. 
If $\mathcal{F}$ is weakly positively curved 
in the sense of Definition \ref{def-weak} and if we have $c_1 (\mathcal{F})=0$, 
then $\mathcal{F}$ is a numerically flat vector bundle on $Y$.
\end{prop}

\begin{proof}
The proof is given by the induction on the rank of $\mathcal{F}$. 
Reflexive sheaves of rank one are always line bundles 
(for example see \cite[Prop 1.9]{Har}), 
and thus the conclusion is obvious in the case of rank one. 

Let $Y_0$ be the maximal locally free locus of $\mathcal{F}$. 
Note that $\codim (Y \setminus Y_0) \geq 3$. 
The main idea of the proof is that, after we take a stable filtration of $\mathcal{F}$, 
we extend quotient sheaves to vector bundles on $Y$, 
by applying \cite{CH} and \cite[Cor 3]{BS}.

To begin with, we fix an ample line bundle $A$ on $Y$ and  we take an $A$-stable filtration of $\mathcal{F}$:
$$
0 \rightarrow \mathcal{F}_1 \rightarrow \mathcal{F}_2 \rightarrow \cdots \rightarrow \mathcal{F}_m := \mathcal{F}.
$$
%
All the sheaves $\mathcal{F}_i$ can be assumed to be reflexive 
by taking the double dual if necessarily. 
We consider the following exact sequence of sheaves:
\begin{align}\label{exact}
0 \rightarrow \mathcal{F}_{m-1} \rightarrow \mathcal{F} \rightarrow 
\mathcal{Q}:=\mathcal{F}_{m}/\mathcal{F}_{m-1}  \rightarrow 0.
\end{align}
As  $\mathcal{F}_m= \mathcal{F}$ is weakly positively curved, 
the quotient $\mathcal{Q}:=\mathcal{F}_{m}/\mathcal{F}_{m-1}$ is also weakly positively curved. 
In particular, its first Chern class $c_1 (\mathcal{Q})$ is  pseudo-effective. 
On the other hand, we have 
$$
0=c_{1}(\mathcal{F})=c_1 (\mathcal{F}_{m-1}) +c_1 (\mathcal{Q})
$$
by the assumption. 
Together with the stability condition of the filtration, 
we can easily check that $c_1 (\mathcal{F}_{m-1})  =0$ and $c_1 (\mathcal{Q})=0$.

We first show that $\mathcal{F}_{m-1}$ is a vector bundle on $Y_0$ and 
that the morphism $\mathcal{F}_{m-1} \rightarrow \mathcal{F}$ is a bundle morphism on $Y_0$, 
following the argument in \cite[Prop 1.16, Cor 1.19]{DPS94}. 
Thanks to \cite[Cor 1.19]{DPS94}, 
it is enough for this purpose to check that the induced morphism
$$
\det \mathcal{Q}^{\star}\to \Lambda^{p} \mathcal{F}^\star, 
$$
is an injective bundle morphism on $Y_0$, 
where $p$ is the rank of $\mathcal{Q}$. 
This morphism determines the section 
$$
\tau \in H^{0}(Y, \Lambda^{p} \mathcal{F}^\star\otimes \det \mathcal{Q}^{\star\star}). 
$$
By $c_{1}(\mathcal{Q})=0$ and the assumption for $\mathcal{F}$, 
it can be seen that  
$\Lambda^{p} \mathcal{F}\otimes \det \mathcal{Q}^{\star}$ is weakly positively curved. 
Then it can be shown that 
$\tau$ is non-vanishing on $Y_0$
by the same argument as in \cite[Prop 1.16]{DPS94}. 
This implies that $\mathcal{F}_{m-1}$ can be seen as a subbundle of $\mathcal{F}$
by the morphism $\mathcal{F}_{m-1} \rightarrow \mathcal{F}$ on $Y_0$. In particular, it follows that the quotient $\mathcal{Q}$ is also  a vector bundle on $Y_0$ from this observation.

On the other hand, we have the surjective bundle morphism 
$$
\Lambda^{r-s+1} \mathcal{F} \otimes \det \mathcal{Q}^\star\to \mathcal{F}_{m-1}
$$
on $Y_{0}$, 
where $r$ (resp. $s$) is the rank of $\mathcal{F}$ (resp. $\mathcal{F}_{m-1}$). 
It follows that 
$\Lambda^{r-s+1} \mathcal{F} \otimes \det \mathcal{Q}^\star$ is weakly positively curved 
from $c_{1}(\mathcal{Q})=0$ and the assumption for $\mathcal{F}$. 
Then the quotient metric on $\mathcal{F}_{m-1}$ is weakly positively curved on $Y_0$
and it can be extended to the maximal locally free locus of $\mathcal{F}_{m-1}$ 
by $\codim (Y \setminus Y_{0}) \geq 3$. 
Therefore we can conclude that $\mathcal{F}_{m-1}$ 
is a numerically flat vector bundle on $Y$ 
by the induction hypothesis.

We now prove that the double dual $\mathcal{Q}^{\star \star}$ of the quotient sheaf 
$\mathcal{\mathcal{Q}}=\mathcal{F}_{m} / \mathcal{F}_{m-1}$ 
is a hermitian flat vector bundle on $Y$.  
Let $S$ be a general surface constructed by the complete intersection 
of hypersurfaces in the complete linear system $|A|$. 
By $\codim (Y\setminus Y_0) \geq 3$, 
we may assume that $S$ is contained in $Y_0$. 
Then the restriction $\mathcal{Q} |_S$ is a weakly positively curved vector bundle 
with $c_{1}(\mathcal{Q}|_S)=0$, 
and thus $\mathcal{Q}|_S$ is numerically flat by \cite[Cor 2.12]{CH}.
Hence it follows that 
$$
\int_S c_2 (\mathcal{Q}|_S)=0
$$
since all the Chern classes of numerically flat vector bundles are zero. 
As a consequence, we can see that 
$$
c_2 (\mathcal{Q}^{\star\star}) \cdot c_1(A)^{n-2}=\int_S c_2 (\mathcal{Q}|_S) =0, 
$$
where $n$ is the dimension of $Y$. 
The double dual $\mathcal{Q}^{\star\star}$ is an $A$-stable reflexive sheaf 
satisfying that $c_1(\mathcal{Q}^{\star\star})=0$ and $c_2 (\mathcal{Q}^{\star\star})\cdot c_1(A)^{n-2}=0$, 
and thus it can be seen that $\mathcal{Q}^{\star \star}$ a hermitian flat vector bundle on $Y$, 
by applying Bando-Siu \cite[Cor 3]{BS} to $\mathcal{Q}^{\star \star}$.

Finally we prove that $\mathcal{F}$ is a vector bundle on $Y$. 
By the above argument, 
sequence (\ref{exact}) is an exact sequence of vector bundles on $Y_0$
and the quotient $\mathcal{Q}$ can be extended to a vector bundle $\mathcal{Q}^{\star \star}$ on $Y$. 
We remark that $\mathcal{Q}$ itself can not be expected to be a vector bundle on $Y$ 
since $\mathcal{Q}$ may not be reflexive. 
By applying Lemme \ref{lemm-elementary} to our case, 
it can be seen that the restriction map 
$$
H^1 (Y, \mathcal{Q}^\star \otimes \mathcal{F}_{m-1}) \rightarrow 
H^1 (Y_0, \mathcal{Q}^\star \otimes \mathcal{F}_{m-1}) 
$$ 
is an isomorphism by $\codim (Y \setminus Y_0) \geq 3$.
Hence the extension class in $H^1 (Y_0, \mathcal{Q}^\star \otimes \mathcal{F}_{m-1}) $
obtained from exact sequence (\ref{exact}) of vector bundles on $Y_0$ 
can be extended to the extension class in $H^1 (Y, \mathcal{Q}^\star \otimes \mathcal{F}_{m-1}) $. 
The extended class determines the vector bundle whose restriction to  $Y_0$  
corresponds to $\mathcal{F}$ by the construction. 
It follows that this vector bundle coincides with $\mathcal{F}$ on $Y$ 
since $\mathcal{F}$ is reflexive (see \cite[Prop 16]{Har}). 
\end{proof}

By using the notion of numerical flatness, 
we obtain a criteria for local triviality with respect to a log pair. 
The following criteria, 
which is essentially proved by the same argument as in \cite{Cao}, 
can be regarded as a pair version of \cite[Prop 2.8]{Cao}. 
For the reader's convenience and the completeness, 
we briefly recall the argument in \cite{Cao} and give a precise proof here.

\begin{prop}[{cf. \cite[Prop 2.8]{Cao}}]\label{isotri}
Let $p: X\rightarrow Y$ be a flat morphism with connection fiber 
between two compact K\"{a}hler manifolds and 
let $\Delta$ be an effective $p$-horizontal divisor on $X$. 
For a $p$-very ample line bundle $L$ on $X$, we set $E_m := p_\star ( m L)$. 
If $E_m$ is numerically flat for every $m \geq 1$ and  
$p_\star (m L \otimes \mathcal{I}_\Delta)$ is numerically flat for some $m$ large enough,
then the pair $(X,\Delta)$ is locally trivial with respect to $p$.

Moreover, for the universal cover $\pi: \widetilde{Y} \rightarrow Y$ of $Y$,  
the fiber product 
$(\widetilde{X}, \widetilde{\Delta}) := (X, \Delta) \times_Y \widetilde{Y}$ admits the following splitting
$$(\widetilde{X}, \widetilde{\Delta} )\cong  \widetilde{Y} \times (F, \Delta |_F),$$
where $F$ is a general fiber of $p$.
\end{prop}

\begin{proof}
As $L$ is $p$-very ample, 
we have the following $p$-relative embedding: 
$$
\xymatrix{
X \ar[rd]_p \ar@{^{(}->}[rr]^j & &\mathbb{P} (E_{1}) \ar[ld]^f\\
& Y}
$$ 
and also we obtain $L = j^\star \mathcal{O}_{\mathbb{P} ( E_1 )} (1)$. 
For $m$ large enough, we have the exact sequence 
\begin{equation}\label{exactseq1pr}
0 \rightarrow f_\star (\mathcal{O} _{\mathbb{P} (E_1)} (m) \otimes \mathcal{I}_X ) \rightarrow 
f_\star ( \mathcal{O} _{\mathbb{P} (E_1)} (m)) \rightarrow p_\star (m L) \rightarrow 0,  
\end{equation}
where $\mathcal{I}_X $ is the ideal sheaf defined by $X \subset \mathbb{P} (E_{1})$.
The vector bundle $E_1$ is a local system as $E_1$ is numerically flat. 
Let $D_{E_1}$ be the flat connection with respect to this local system. 
Put $n:=\rank E_1$. 
Since $\widetilde{Y}$ is simply connected, 
the pull-back $\pi^\star E_1$ is a trivial vector bundle on $\widetilde{Y}$, 
and thus we can take global flat sections (with respect to $D_{E_1}$) 
$$
\{ e_1, e_2, \cdots, e_n \} \subset H^0 (\widetilde{Y}, \pi^\star E_1)
$$ 
such that $\{ e_1, e_2, \cdots, e_n \} $ generates $\pi^\star E_1$.

\smallskip

Set $ F_m := f_\star (\mathcal{O} _{\mathbb{P} (E_1)} (m) \otimes \mathcal{I}_X)$. 
Then $F_m$ is also numerically flat, since 
both $f_\star ( \mathcal{O} _{\mathbb{P} (E_1)} (m)) =\Sym^m E_1$ and $p_\star (m L)$ 
are numerically flat by the assumptions. 
Hence we can take the flat connection $D_{F_m}$ on the local system $F_m$. 
Then, by the same way as above, we can see that 
$\pi^\star F_m$ is a trivial vector bundle 
and we can take global flat sections (with respect to $D_{F_m}$) 
$$\{ s_1, s_2, \cdots, s_t \} \subset H^0 (\widetilde{Y}, \pi^\star F_m)$$ 
such that $\{ s_1, s_2, \cdots, s_t \} $ generates $\pi^\star F_m$, where $t$ is the rank of $F_m$.

Let $$
\varphi: \pi^\star f_\star (\mathcal{I}_X \otimes \mathcal{O} _{\mathbb{P} (E_1)} (m)) \rightarrow 
\pi^\star f_\star ( \mathcal{O} _{\mathbb{P} (E_1)} (m))=\pi^\star \Sym^m E_1
$$ 
be the inclusion induced by \eqref{exactseq1pr}.
Let $D_{\Sym^m E_1}$ be the flat connection on $\Sym^m \pi^\star E_1$ induced by $D_{E_1}$.
Thanks to \cite[Lemma 4.3.3]{Cao13}, 
we know that  
the section $\varphi (s_i)$ is flat with respect to the connection $D_{\Sym^m E_1}$ 
for every $i$.
In particular, for every $i$, we can find constants $a_{i, \alpha}$ such that 
$$\varphi (s_i)= \sum_{\alpha= (\alpha_1, \cdots, \alpha_n), |\alpha|=m} a_{i, \alpha} \cdot e_1 ^{\alpha_1} e_2^{\alpha_2}\cdots e_n^{\alpha_n}.$$
In other words, the $p$-relative embedding of $\widetilde{X}$ 
into $\mathbb{P}^{n-1} \times   \widetilde{Y}$:
$$
\xymatrix{
\widetilde{X} \ar[rd]_{\tilde p} \ar@{^>}[rr] & &\mathbb{P}^{n-1} \times   
\widetilde{Y} \ar[ld]^{\rm{pr}_2}\\
& \widetilde{Y}}
$$ 
is defined by the polynomials $\varphi (s_i)$ 
whose coefficients are independent of $y \in \widetilde{Y}$.
Then $p$ is locally trivial and we obtain the splitting
$\widetilde{X}\simeq \widetilde{Y} \times F$,
where $F$ is the general fiber of $p$. 

\smallskip

To identify the pairs, let $p_\Delta : \Delta \rightarrow Y$ be 
the restriction of $p$ to $\Delta$. For $m$ large enough, we have the exact sequence
\begin{equation}\label{exactseq1prpair2}
 0 \rightarrow  p_\star (mL \otimes \mathcal{I}_\Delta) \rightarrow p_\star (m L) \rightarrow   (p_\Delta)_\star (m L |_\Delta) \rightarrow  0.
\end{equation}
Since the first two terms are numerically flat, 
the third term $(p_\Delta)_\star (m L |_\Delta)$ is also numerically flat.
By taking a sufficiently large  $m$, 
the sheaf $(p_\Delta)_\star (m L |_\Delta)$ fits in another the exact sequence 
\begin{equation}\label{exactseq1prpair}
0 \rightarrow f_\star (\mathcal{I}_\Delta \otimes \mathcal{O} _{\mathbb{P} (E_1)} (m)) \rightarrow 
f_\star ( \mathcal{O} _{\mathbb{P} (E_1)} (m)) \rightarrow (p_\Delta)_\star (m L |_\Delta) \rightarrow 0. 
\end{equation}
As we proved that $(p_\Delta)_\star (m L |_\Delta)$ is numerically flat, 
the first term $f_\star (\mathcal{I}_\Delta \otimes \mathcal{O} _{\mathbb{P} (E_1)} (m))$ is numerically flat.

By using the same argument for 
$f_\star (\mathcal{I}_X \otimes \mathcal{O} _{\mathbb{P} (E_1)} (m))$ as above, 
we know that the $p$-relative embedding of $\widetilde{\Delta}$ in $\mathbb{P}^{n-1} \times   \widetilde{Y}$ 
is defined by the polynomials whose coefficients are independent of $y \in \widetilde{Y}$.  
As a consequence,  we can conclude that 
$p: (X, \Delta) \rightarrow Y$ 
is locally trivial and we have the splitting
$(\widetilde{X}, \widetilde{\Delta} )\cong  \widetilde{Y} \times ( F, \Delta |_F)$ 
for a general fiber $F$ of $p$. 
\end{proof}


\subsection{Sections of certain flat vector bundles}\label{Sec2-3}
In this subsection, 
we prove the following proposition, 
which is needed to estimate the dimension of global sections of certain flat vector bundles 
in the proof of Theorem \ref{thm-main}.

\begin{prop}\label{prop-rank}
Let $\pi: \Gamma \to X$ be a birational morphism from 
a smooth projective variety $\Gamma$ to a 
$($not necessarily smooth$)$ projective variety $X$, 
and let $V_m$ be a torsion free sheaf on $\Gamma$. 
Assume that the restriction $V_m |_{{C}}$ of $V_m$ to a smooth curve ${C}$ 
is a numerically flat vector bundle on ${C}$, 
where ${C}$ is a complete intersection 
$$
{C}:={H}_{1} \cap {H}_{2} \cap \dots \cap {H}_{\dim X -1}
$$  
constructed by a general member ${H}_{i}$ in the free linear system $|\pi^{\star}A|$ 
for an arbitrary very ample line bundle $A$  on $X$. 
Then, for any line bundle $L$ on $\Gamma$, 
there exists a constant $C_L$ depending only on $L$ and $X$ $($independent of $V_m$$)$ such that 
$$
h^{0}(X, V_m \otimes L) \leq C_L \cdot \rank V_m. 
$$
\end{prop}

Before proving Proposition \ref{prop-rank}, 
we first give a proof for the following lemma, 
which can be regarded as a vanishing theorem of Fujita type 
for numerically flat vector bundles. 

\begin{lemm}\label{lemma-vanishing}
Let $A$ be an ample line bundle on a smooth projective variety $X$ 
and $E$ be a numerically flat vector bundle on $X$. 
Then we have 
$$
H^{q}(X, K_{X}\otimes E \otimes A)=0
$$
for any $q>0$. 
\end{lemm}

\begin{rem}\label{rem-vanihsing}
It is natural to ask whether or not 
the same conclusion holds under the weaker assumption that $E$ is nef, 
since it is true in the case where $E$ is a line bundle by the Kodaira vanishing theorem. 
When $X$ is a smooth curve, we can easily check it. 
Indeed, there is a finite morphism $p: X' \to X$ such that 
$p^\star(E\otimes A)$ is the quotient of the direct sum $\oplus_i B_i$ of 
ample line bundles $B_i$  since $E\otimes A$ is an ample vector bundle (see \cite{CF90}). 
Then we can easily check that  
\begin{align*}
0=h^{1}(X', K_{X'}\otimes p^\star(E\otimes A))
&=h^{0}(X', p^\star(E^\star\otimes A^\star))\\
&=h^{0}(X, E^\star\otimes A^\star\otimes p_\star \mathcal{O}_{X'})\\
&\geq h^{0}(X, E^\star\otimes A^\star) \\
&\geq h^{1}(X, K_X \otimes E \otimes A)
\end{align*}
by the Serre duality, the projection formula, and the injectivity $\mathcal{O}_X \hookrightarrow p_\star \mathcal{O}_{X'}$. 

\end{rem}

\begin{proof}
We first observe that the conclusion is obvious if $E$ is a hermitian flat vector bundle  
(that is, it admits a smooth hermitian metric such that $\sqrt{-1}\Theta_h(E)=0$). 
Indeed, the vector bundle $E \otimes A$ is Nakano positive 
since the metric $h \otimes h_A$ satisfies that 
$$
\sqrt{-1}\Theta_{h \otimes h_A}(E\otimes A)= {\rm{id}}_E \otimes \sqrt{-1}\Theta_h(A) >_{\rm{Nak}} 0, 
$$
where $h_A$ is a smooth hermitian metric on $A$ with positive curvature. 
Therefore, in the case of $E$ being hermitian flat, 
we can obtain the desired conclusion by Nakano's vanishing theorem. 

Now we treat the general case of $E$ being numerically flat. 
Then, by Theorem \ref{thm-numflat},  
there exists a filtration of $E$ by subbundles $\{E_k\}_{k=0}^{p}$
$$
0=:E_0 \subset E_1 \subset \dots \subset E_{p-1} \subset E_p:=E 
$$
such that the quotient bundle $E_k/E_{k-1}$ is a hermitian flat vector bundle. 
It follows that 
$$
H^{q}(X, K_X \otimes E_k/E_{k-1} \otimes A)=0
$$
for any $q \geq 1$ and $1 \leq k \leq p$ from the first observation. 
Hence we can easily check that 
$$
h^{q}(X, K_X \otimes E_{k} \otimes A) \leq 
h^{q}(X, K_X \otimes E_{k-1} \otimes A) 
$$
by the long exact sequence obtained from the exact sequence 
$$
0 \to 
K_X \otimes E_{k-1} \otimes A \to 
K_X \otimes E_{k} \otimes A \to 
K_X \otimes E_k/E_{k-1} \otimes A \to 0. 
$$
The above inequalities and the hermitian flatness of $E_1$ yield  
$$
h^{q}(X, K_X \otimes E \otimes A)= h^{q}(X, K_X \otimes E_{p} \otimes A) \leq 
h^{q}(X, K_X \otimes E_{1} \otimes A)=0. 
$$
This is the desired conclusion. 
\end{proof}

%

\begin{proof}[Proof of Proposition \ref{prop-rank}]
For a given line bundle $L$ on $\Gamma$, we take a sufficiently ample line bundle $B$ on $X$ such that 
$L^\star \otimes \pi^\star B$ is a big line bundle on $\Gamma$. 
We first show that 
$$
H^0(\Gamma, V_m \otimes L \otimes \pi^\star B^\star )=0
$$
by contradiction. 
For a general hyperplane ${H}$ in the free linear system $|\pi^\star  A|$ and 
a non-zero section 
$$
t\in H^0 (\Gamma, V_m \otimes L \otimes \pi^\star B^\star),$$ 
it is easy to see that 
the restriction $t|_{{H}}$ is a non-zero section of 
$V_m \otimes L \otimes \pi^\star B^\star|_{{H}}$ 
and that the restriction $L^\star  \otimes \pi^\star B |_{{H}}$ is a big line bundle on ${H}$.  
By repeating this process for general hypersurfaces in $|\pi^\star  A|$, 
we can construct a smooth curve ${C}$ with the following properties\,$:$
\vspace{0.1cm}\\
\quad $\bullet$ ${C}$ is a complete intersection constructed by general members in $|\pi^\star  A|$. \\
\quad $\bullet$ $t|_{{C}}$ is a non-zero section of 
$V_m \otimes L \otimes \pi^\star B^\star|_{{C}}$. \\
\quad $\bullet$ $L^\star  \otimes \pi^\star B |_{{C}}$ is still big (that is, ample) 
on ${C}$. 
\vspace{0.1cm}\\
Then, by Lemma \ref{lemma-vanishing} and the Serre duality, 
we have 
$$
h^{0}({C}, V_m \otimes L \otimes \pi^\star B^\star |_{{C}})
=h^{1}({C}, K_C \otimes (V_m|_{{C}})^\star  \otimes L^\star  \otimes \pi^\star B|_{{C}})
=0. 
$$
Here we used the assumption that $V_m|_{{C}}$ is numerically flat. 
This contradicts the existence of the non-zero section $t|_{{C}}$. 

Now we consider a general hypersurface ${H}$ in $|\pi^\star B|$ (not $|\pi^\star A|$). 
It follows that the following sequence 
$$
0 \to 
V_m \otimes L \otimes \pi^\star B^\star  \to 
V_m \otimes L  \to 
V_m \otimes L |_{{H}} \to 0. 
$$
is exact since $V_{m}$ is torsion free. 
By the first half argument and the induced long exact sequence, 
we obtain  
$$
h^{0}(X, V_m \otimes L) \leq 
h^{0}({H}, V_m \otimes L |_{{H}}). 
$$
We may assume that 
$L^\star  \otimes \pi^\star  B |_{{H}}$ is a big line bundle on ${H}$, 
and thus we obtain 
$$
h^{0}(X, V_m \otimes L) \leq 
h^{0}({H}, V_m \otimes L |_{{H}}) \leq 
h^{0}({H} \cap {H'}, V_m \otimes L |_{{H}\cap {H'}})
$$
for a general hypersurface ${H}$ in $|\pi^\star B|$ 
by the same argument as above. 
By repeating this process,  
it can be shown that  
$$
h^{0}(X, V_m \otimes L) \leq 
h^{0}({C}, V_m \otimes L |_{{C}})  
\leq h^{0}(\Sigma, V_m \otimes L |_{\Sigma})= \sharp \Sigma \cdot \rank V_m. 
$$
Here ${C}$ is a curve constructed by 
a general complete intersection of members in $|\pi^\star B|$ and 
$\Sigma$ is the set of the reduced points defined by ${C} \cap H$, 
where ${H}$ is a general hypersurface in $|\pi^\star B|$. 
The cardinality $\sharp \Sigma$ of $\Sigma$ is equal to the self intersection number of $\pi^\star B$, 
and thus it depends only on $B$ (but not $V_m$). 
This finishes the proof. 
\end{proof}

In the rest of this subsection, 
we observe the following famous example 
(see \cite[Example 1.7]{DPS94} for more details),
which tells us that the equality in Theorem \ref{thm-main} does not hold for the Kodaira dimension  
even when $X$ is smooth and $\Delta$ is the zero divisor. 

\begin{ex}\label{exa}
Let $Y$ be an elliptic curve and $E$ be the (uniquely determined) 
vector bundle $E$ on $Y$ having the non-splitting exact sequence  
$$
0 \to \mathcal{O}_Y \to E \to  \mathcal{O}_Y \to 0. 
$$
We consider the projective space bundle $\varphi: X:=\mathbb{P}(E) \to Y$. 
Then it is easy to see that the anti-canonical divisor $-K_X$ is nef (but not semi-ample). 
Also it can be seen that 
\begin{align*}
\kappa(-K_X)=0 < 1=\kappa(-K_{X_y}) \text{ and }
\nd(-K_X) = \nd(-K_{X_y})=1.  
\end{align*}
Here $X_y$ is a fiber of the MRC fibration $\varphi: X \to Y$.

Moreover it can be shown that $E$ is numerically flat, but not hermitian flat. 
The direct image sheaves $V_m$ defined in Section \ref{Sec3} 
are calculated by 
$$
V_{m}=\varphi_\star(-mK_{X/Y})=\varphi_\star(-mK_{X})={\rm{S}}^{m}E 
$$
in this example. 
This example says that we actually need to treat numerically flat 
(not hermitian flat) vector bundles. 
\end{ex}

\section{Positivity of direct images of nef relative anti-canonical divisors}\label{Sec3}
In this section, we are always in the following setting. 
The special case in the following setting will be discussed in Section \ref{Sec4}.  

\begin{set}\label{setting}
Let $\psi: X \dashrightarrow Y$ be an almost holomorphic map 
from a normal projective variety $X$ to a smooth projective variety $Y$.
We choose a desingularization $\holom{\pi}{\Gamma}{X}$ such that 
it is a resolution of the indeterminacy locus of $\psi$: 
\begin{align}\label{diagram}
\xymatrix{
\Gamma \ar[rd]_\varphi  \ar[rr]^\pi & & X \ar@{-->}[ld]^\psi
\\
& Y  &}
\end{align}
We assume in this setting that there exists an effective klt $\mathbb{Q}$-divisor $D$ on $\Gamma$ such that 
\begin{equation}\label{mainassump}
-\pi_\star (c_1 (K_{\Gamma/Y} +D)) \text{ is a nef $\mathbb{Q}$-divisor}. 
\end{equation}
To begin with, let us fix some notations. 
Let $E$ be the $\pi$-exceptional locus, namely, we denote the notation $E$ by 
$$\text{
$E:=$the reduced divisor whose support coincides with the $\pi$-exceptional locus.
}
$$
Thanks to \eqref{mainassump}, there exists some $\pi$-exceptional $\mathbb{Q}$-divisor $E''$ such that
\begin{equation}\label{nefconddd}
-(K_{\Gamma/Y} + D) + E''  \text{ is a nef $\mathbb{Q}$-divisor}.
\end{equation}
Note that if $X$ is smooth, then $\varphi (E) \subsetneq Y$. However, in general, it is possible that 
$$\varphi (E) = \varphi (E'')=Y.$$
Now we choose a sufficiently ample line bundle $A$ on $\Gamma$ satisfying that  
the morphism 
\begin{equation}\label{surgene}
\mbox{\rm Sym}^p H^0 (\Gamma_y, A|_{\Gamma_y}) \rightarrow H^0 (\Gamma_y, p A|_{\Gamma_y})
\end{equation}
is surjective for every $p \in \mathbb{Z}_+$ and for a general fiber $\Gamma _y$. 

Let $Y_0 \subset Y$ be the largest Zariski open subset with the following properties\,$:$
\vspace{0.1cm}\\
\qquad $\bullet$  The induced morphism $\Gamma_0 := \fiber {\varphi}{Y_0} \rightarrow Y_0$ is equi-dimensional.\\
\qquad $\bullet$  The pull-back $\varphi^\star B$ is not contained in $E$ 
for any prime divisor $B \subset Y_0$. 
\end{set}
The following result proved in \cite{Zha05} is one of key propositions. 

\begin{theo}[{\cite[Main thm]{Zha05}}]\label{prop-fun}
Under setting $($\ref{setting}$)$, 
we further assume that $\psi:X \dashrightarrow Y$ is an MRC fibration. 
Then the followings holds\,$:$
\begin{enumerate}
\item Every irreducible component of $\pi(\fiber {\varphi}{Y \setminus Y_0})$ has codimension at least two. In particular, every $\varphi$-exceptional divisor is $\pi$-exceptional.

\item Let  $\varphi^\star B = \sum_{j=1}^l m_j B_j$ be the irreducible decomposition for a prime divisor $B$ on $Y_0$. If $m_j>1$, then the divisor $B_j$ is $\pi$-exceptional. 
\end{enumerate}
\end{theo}

\begin{rem}\label{rem-set}

For a given klt pair $(X, \Delta)$ with the nef anti-canonical divisor $-(K_X+\Delta)$. Let $X \dashrightarrow Y$ be the MRC fibration.
Thanks to \cite{Zha05}, we know that $\kappa (Y) =0$. 
We take a log resolution $\holom{\pi}{\Gamma}{X}$ of $(X, \Delta)$ and $X \dashrightarrow Y$ as in \eqref{diagram}.
Let $N$  be the $\mathbb{Q}$-effective divisor on $Y$ such that 
$N$ is $\mathbb{Q}$-linearly equivalent to $K_Y$. 
From the proof of \cite[Main thm]{Zha05}, we can see that $\supp \varphi^{\star}(N)$ is $\pi$-contractible.

Now we check that $X \dashrightarrow Y$ satisfies the setting \ref{setting}.
We have the canonical bundle formula 
$$
K_{\Gamma}+\Delta'=\pi^\star (K_X+\Delta)+E', 
$$
where $\Delta'$ and $E'$ are the effective divisors with no common components, 
and $\Delta'$ is a klt divisor. 
Then we have 
$$
-(K_{\Gamma/Y}+\Delta')=\varphi^{\star}N - E' - \pi^{\star}(K_X+\Delta). 
$$
The assumption (\ref{mainassump}) is satisfied by the nefness of $-(K_X+\Delta)$ and the fact that $\supp \varphi^{\star}(N)$ is $\pi$-contractible.
Further if we define $E'':=  E'-\varphi^{\star}N$, 
then the condition \eqref{nefconddd} is also satisfied by the above theorem. 
\end{rem}


The following two lemmas also play an important role in the proof, 
and they follow from the arguments in \cite[lemma 3.5, lemma 3.6]{CH}.

\begin{lemm}\label{denew}
Let $L$ be a pseudo-effective and $\varphi$-big line bundle on $\Gamma$. 
Let $p\in \mathbb{Z}$ be an $($not necessarily positive$)$ integer. 
If $\varphi_\star (p (K_{\Gamma/Y}+D) +L)$ is a non-zero sheaf,
then $\varphi_\star (p (K_{\Gamma/Y}+D) +L + c E'')$ is weakly positively curved on $Y$ 
for any $c$ large enough with respect to $p$. 
\end{lemm}

\begin{proof}
Let $m$ be a positive integer such that $m (D -E'')$ is Cartier. 
By condition \eqref{nefconddd}, we can see that 
\begin{equation}\label{birnef}
M: = - m (K_{\Gamma/Y} + D) +m E'' \text{ is nef},
\end{equation}
and also we have
\begin{equation}\label{equll}
p (K_{\Gamma/Y}+D) + L + m E'' =(p+m)(K_{\Gamma/Y}+D) + (M+ L). 
\end{equation}
Thanks to Theorem \ref{thm-cur}, for $m$ large enough (with respect to $p$ and $L$), 
$$\varphi_\star \left( (p+m)(K_{\Gamma/Y}+D) + (M+ L) \right)$$ 
is weakly positively curved on $Y$. Together with \eqref{equll}, 
the sheaf $\varphi_\star (p (K_{\Gamma/Y}+D) +L + m E'')$ is weakly positively curved.
\end{proof}

\begin{lemm} \label{prop-fun2}
Let $L$ be a $\varphi$-relatively big line bundle on $\Gamma$.
Then, for every $m \in \mathbb{Z}$,  
the direct image sheaf 
$\pi_\star (L  - \frac{1}{r}\varphi^\star (\det \varphi_\star (L+ m E))$ is pseudo-effective, 
where $r$ is the rank of $\varphi_\star (L+ m E)$.
\end{lemm} 

\begin{proof}
Let $E_1 \subset \Gamma$ be the union of the non-reduced and non-flat locus of  $\varphi$. 
Let $\Gamma^{(r)}$ be a desingularization of the $r$-times $\varphi$-fiberwise 
product $\Gamma \times_Y \cdots \times_Y \Gamma$. 
Let $\pr_i : \Gamma^{(r)} \rightarrow \Gamma$ be the $i$-th directional projection, 
and let $\varphi_r : \Gamma^{(r)} \rightarrow Y$ be the natural fibration.  
Set $L^{(r)} := \sum_{i=1}^r  \pr_i ^\star L$ and $E^{(r)} :=\sum_{i=1}^r \pr_i ^\star  E$.
We have a natural morphism
\begin{equation}\label{natmor}
\det \varphi_\star (\sO_\Gamma(L +mE)) \rightarrow (\varphi_r)_\star (\sO_{\Gamma^{(r)}}(L^{(r)} + m E^{(r)} +E_2))
\end{equation}
for some divisor $E_2$ supported in $\sum\limits_{i=1}^r \pr_i ^\star (E_1)$.
Set 
$$
L' :=L^{(r)} + m E^{(r)} +E_2 - \varphi_r ^\star c_1(\varphi_\star (\sO_\Gamma(L +mE))).
$$
Since the restriction of \eqref{natmor} on a general point $y\in Y$ is non zero, 
the morphism \eqref{natmor} induces a nontrivial section in 
$H^0 (\Gamma^{(r)}, L')$. Therefore $L'$ is an effective line bundle on $\Gamma^{(r)}$. 

\medskip

Let $\varphi^{(r)} : \Gamma^{(r)}\rightarrow Y$ be the natural morphism and set $\Delta^{(r)} := \sum_{i=1}^r \pr_i ^\star D$.
We can thus find a (not necessary effective) divisor  $E_3$ supported in 
$\sum_{i=1}^r \pr_i ^\star (E_1 +E)$ such that
$$
- K_{\Gamma^{(r)}/Y}  - \Delta^{(r)} +E_3  \text{ is nef}.
$$
By using \cite[Proposition 2.10]{Cao}, there exists an ample line bundle $A_Y$ on $Y$ such that for every $p\in\mathbb{Z}_+$ and for $q$ large enough (with respect to $p$),
the restriction
\begin{align*}
&H^0 (\Gamma^{(r)}, pq (K_{\Gamma^{(r)} /Y}+  \Delta^{(r)} )+ pq (- K_{\Gamma^{(r)}/Y} - \Delta^{(r)}+E_3) + p L' +\varphi_r ^\star A_Y ) \\
\rightarrow 
&H^0 (\Gamma^{(r)} _y,  pq E_3 + p L' +\varphi_r ^\star A_Y ) 
\end{align*}
is surjective for a general $y \in Y$.

By restricting the above morphism on the diagonal of $\Gamma^{(r)}$
(which is well-defined only on $\Gamma\setminus E_1$, 
but we can extend it on $\Gamma$ by adding some divisor supported in $E_1$), 
we can find some effective divisor $F_{p, q}$ (depending on $p$ and $q$) supported in $E \cup  E_1$ such that
\begin{align*}
&H^0 (\Gamma,   F_{p, q} + p r L  -  p \varphi^\star c_1(\varphi_\star (\sO_\Gamma(L+mE))) +\varphi^\star A_Y )\\
\rightarrow 
&H^0 (\Gamma_y,   F_{p, q} + p rL  - p \varphi^\star c_1(\varphi_\star (\sO_\Gamma(L+mE))) +\varphi^\star A_Y ) 
\end{align*}
is surjective.
In particular, it can be see that 
$$
F_{p, q} + p r L  - p \varphi^\star c_1(\varphi_\star (\sO_\Gamma(L+mE))) +\varphi^\star A_Y
$$ is effective.
Thanks to Theorem \ref{prop-fun}, 
the subvariety 
$\pi_\star (F_{p, q}) \subset \pi_\star (E \cup E_1)$ has codimension at least two. 
Therefore we can conclude that 
$$
\pi_\star 
\left(
c_1
\left(
r L  - \varphi^\star c_1(\varphi_\star (\sO_\Gamma(L+mE))) +\frac{1}{p}\varphi^\star A_Y)
\right)
\right)
$$ 
is pseudo-effective on $X$.
The lemma is proved by letting $p \rightarrow +\infty$.
\end{proof}

We need also the following lemma, which is very close to \cite[Lemma 3.7]{CH}. 

\begin{prop} \label{prop-fun3}
There exists a positive integer $m_0 >0$ 
such that $\varphi_\star (A+ m E)$ is of the same rank and the natural morphism 
\begin{equation}\label{isomy}
\det \varphi_\star (A+mE) \rightarrow \det \varphi_\star (A+(m+1)E) \text{ on } Y_0, 
\end{equation}
is an isomorphism for every $m \geq m_0$.

In particular, if we set $\widetilde{A}$ by  
$$
\widetilde{A} := A +m_0 E - \frac{1}{r_{m_0}} \varphi^\star \det \varphi_\star (A+ m_0 E)
$$ 
then $\pi_\star \widetilde{A}$ is pseudo-effective and $\pi_\star \varphi^\star c_1 \big(\varphi_\star (\widetilde{A} + \widetilde{E})\big) =0$ for any effective $\pi$-exceptional divisor $\widetilde{E}$.
Here $r_{m_0}$ is the rank of $\varphi_\star (A+ m_0 E)$.
\end{prop} 

\begin{proof}
We first prove that 
$\varphi_\star (A+mE)$ is of constant rank for a sufficiently large $m$. 
Let $A_Y$ be a sufficiently very ample line bundle on $Y$.
Recall that the $\pi$-exceptional divisor $E''$ satisfies condition (\ref{nefconddd}). 
By the standard Ohsawa-Takegoshi extension theorem, the restriction map 
\begin{align*}
&H^0 (\Gamma, A  +mE + pE''+ \varphi^\star A_Y) \\
= &H^0 (\Gamma, p( K_{\Gamma/Y} + D) + p( - K_{\Gamma/Y} - D +E'') + A +mE + \varphi^\star A_Y)\\
\longrightarrow &H^0 (\Gamma_y, p( K_{\Gamma/Y} + D) + p( - K_{\Gamma/Y} - D +E'') + A +mE)
\end{align*}
is surjective 
for a general $y\in Y$ and for  $p \gg m$.
The support $E+E''$ is contained in the $\pi$-exceptional locus, and thus 
$$h^0 (\Gamma,  A + m E + p E''   + \varphi^\star A_Y)$$
is independent of $p$ and $m$ when $p\gg m\gg 1$. 
Then the above subjectivity implies that the rank of $\varphi_\star (A+ mE + p E'') $ 
can be estimated by a constant when when $p \gg m \gg 1$. 
Then we have 
$$
\rank \varphi_\star (A+ mE ) \leq \rank \varphi_\star (A+ mE + p E'') \leq C.
$$
Therefore the rank of $\varphi_\star (A+ m E)$ is constant when $m$ is large enough.

\smallskip

We now prove that the morphism \eqref{isomy} is an isomorphism on $Y_0$ by contradiction.
Let $m_0 \in\mathbb{Z}_+$ such that $\rank \varphi_\star (A+mE)$ is constant for $m \geq m_0$.
Then it can be seen that 
$$
\det \varphi_\star (A+(m+1)E)  - \det \varphi_\star (A+mE) \text{ is pseudo-effective}
$$ for every $m \geq m_0$.
If \eqref{isomy} is not an isomorphism on $Y_0$ for infinitely many $m\in\mathbb{Z}_+$, we can find a sequence $a_k \rightarrow +\infty$ and some effective divisor
$D_k$ with $D_k \cap Y_0 \neq \emptyset$ such that
$$\det \varphi_\star (A+(a_k+1)E)  \geq \det \varphi_\star (A+ a_k E) + [D_k].$$ 
However, by Proposition \ref{prop-fun2}, we have 
$$\pi_\star (A) \geq \frac{1}{r}\pi_\star (\varphi^\star\det \varphi_\star (A+ m E)) \text{ for every } m\in\mathbb{Z}_+.$$
Therefore we obtain 
$$\pi_\star (A) \geq \frac{1}{r} \pi_\star \varphi^\star\det \varphi_\star (A+ m_0 E) + \frac{1}{r} \sum_{k=1}^\infty \pi_\star \varphi^\star [D_k].$$
This is a contradiction to the fact that 
$\pi_\star \varphi^\star [D_k]$ is an effective non-zero divisor.
Therefore the morphism \eqref{isomy} is an isomorphism over $Y_0$ when $m$ is large enough.

\smallskip

For the latter conclusion of the proposition, 
thanks to Lemma \ref{prop-fun2}, we can see that $\pi_\star \widetilde{A}$ is pseudo-effective.
By the construction, we obtain 
$$
\pi_\star \varphi^\star c_1 \big(\varphi_\star (\widetilde{A})\big) =0. 
$$
By the isomorphism \eqref{isomy} on $Y_0$, we know that 
$$
\det \varphi_\star (A+ m_0 E + \widetilde{E}) \rightarrow \det \varphi_\star (A+ m_0 E) 
$$
is also isomorphic  on  $Y_0$. 
Therefore we obtain the desired conclusion 
$\pi_\star \varphi^\star c_1 \big(\varphi_\star (\widetilde{A} + \widetilde{E})\big) =0$.
\end{proof}

The following proposition, which is obtained from the above lemmas, 
is the main proposition of this section. 
 
\begin{prop} \label{prop-fun5}
Let $\widetilde{A}$ be the line bundle defined in Proposition \ref{prop-fun3}.
Then, for every $m, p\in\mathbb{Z}_+$, by taking $c_{m, p} >0$ large enough, 
the direct image sheaf 
$$
\cV_{m, p} := \varphi_\star (-m (K_{\Gamma/Y} +D) + c_{m, p} E'' + p \widetilde{A})
$$ 
is weakly positively curved and $\pi_\star \varphi^\star c_1 (\cV_{m, p})  =0$.

In particular, for the complete intersection  $\overline{C}$  of 
general members of a very ample line bundle on $X$, 
the restriction $\varphi^\star \cV_{m, p} |_{{C}}$ is numerically flat.
Here ${C}:=\pi^{-1}(\overline{C})$
\end{prop}

\begin{proof}
The line bundle $\widetilde{A}$ is pseudo-effective and $\varphi$-big, 
and thus  Lemma \ref{denew} implies that $\cV_{m, p}$ is weakly positively curved 
for $c_{m, p}$ large enough. 
In particular, its first Chern class $c_1 (\cV_{m, p})$ is pseudo-effective on $Y$. 
Hence it is sufficient for the proof to 
prove that $- \pi_\star \varphi^\star  c_1 (\cV_{m, p})$ is pseudo-effective.
The proof is as follows.

Thanks to Proposition \ref{prop-fun2}, 
there exists some $\pi$-exceptional effective divisor $F$ such that 
$$ 
L: =-m (K_{\Gamma/Y} +D) + c_m E''  + p \widetilde{A} + F -  \frac{1}{r_1} \cdot \varphi^\star \det \cV_{m, p} \text{ is pseudo-effective},
$$
where $r_1$ is the rank of $\cV_{m, p}$.
Applying Lemma \ref{denew} to $L$, we have 
$$
c_1\Big( \varphi_\star ( m (K_{\Gamma/Y} +D) +  d_m  E'' +  L)\Big) \text{ is pseudo-effective}
$$
for some $d_m$ large enough.
It is equivalent to say that  
$$c_1 \big(\varphi_\star (c_m E'' + d_m  E'' + F +p\widetilde{A})\big) -  \frac{r}{r_1} c_1 (\cV_{m, p}) 
\text{ is pseudo-effective on } Y.$$
From Proposition \ref{prop-fun3}, it follows that 
$$
\pi_\star \varphi^\star c_1 \big(\varphi_\star (c_m E'' + d_m  E'' + F + \widetilde{A})\big)  =0.
$$
Therefore $ - \pi_\star \varphi^\star c_1 (\cV_{m, p})$ is pseudo-effective on $X$. 
The proposition is proved.
\end{proof}

\section{Structure of rationally connected fibrations}\label{Sec4}

\subsection{Proof of Theorem \ref{thm-main}}\label{Sec4-1}

This subsection is devoted to the proof of Theorem \ref{thm-main}.

\begin{proof}[Proof of Theorem \ref{thm-main}]

We first show that 
the inequality $\nd (-(K_X+\Delta)) \geq \nd (-(K_{X_y}+\Delta_{X_y}))$ holds 
by proving the following extension theorem 
from a general fiber $X_y$ to the ambient space $X$. 


\begin{prop}\label{prop-ext}
Let $(X, \Delta)$ be a projective klt pair such that $-(K_X+\Delta)$ is nef. 
Let $\psi: X \dashrightarrow Y$ be an almost holomorphic map from $X$ 
to a smooth projective variety $Y$. 
Then there exists an ample line bundle $A$ on $X$ such that the restriction map 
to a general fiber $X_y$ of $\psi$ 
$$
H^{0}(X, -m(K_X+\Delta)+A) \rightarrow H^{0}({X_y}, -m(K_{X_y}+\Delta_{X_y})+A|_{X_y})
$$
is surjective 
for any $m \geq 1$ with $m(K_X+\Delta)$ Cartier. 

In particular, we have  
$$
\nd (-(K_X+\Delta)) \geq \nd (-(K_{X_y}+\Delta_{X_y})) 
$$
for a general fiber $X_y$ of an MRC fibration $\psi: X \dashrightarrow Y$ of $X$. 
\end{prop}

\begin{proof}
Throughout the proof, 
let $L$ be the Cartier $\mathbb{Q}$-divisor defined by $L:=-(K_X+\Delta)$ and 
$m$ be a positive integer such that $mL=-m(K_X+\Delta)$ is a Cartier divisor. 
We take a log resolution $\pi: \Gamma \to X$ of $(X, \Delta)$ such that 
it induces  a resolution $\varphi:\Gamma \to Y$ of the indeterminacy locus $B$ of $\psi$. 
We have the same commutative diagram (\ref{diagram}) as in Section \ref{Sec3}. 
We remark that a general fiber $X_y$ of $\psi$ is compact 
since $\psi$ is an almost holomorphic map. 
For a general fiber $X_y$ of $\psi$, 
the inverse image $\Gamma_y:=\pi^{-1}(X_y)$ is also a general fiber of $\varphi$, 
and thus $\Gamma_y$ is a smooth subvariety in $\Gamma$. 
Further a general fiber $\Gamma_y$ is not contained in 
any components of the exceptional locus of $\pi$. 
Then we can obtain the following commutative diagram 
$$
\xymatrix@C=40pt@R=30pt{
H^{0}(X,mL+A)  \ar[r]^{\pi^\star \quad}_{\cong \quad}  \ar[d]
& H^{0}(\Gamma,\pi^\star (mL+A))  \ar[r]^{+F\quad}_{\cong\quad} \ar[d]
& H^{0}(\Gamma,\pi^\star (mL+A)+F) \ar[d] \\ 
H^{0}(X_y,(mL+A) |_{X_y})  \ar[r]^{(\pi|_{\Gamma_y})^\star \quad}_{\quad} & 
H^{0}(\Gamma_y,\pi^\star (mL+A)|_{\Gamma_y}) 
\ar[r]\ar[r]^{+F|_{\Gamma_y}\quad}
& H^{0}(\Gamma_y,\pi^\star (mL+A)+F|_{\Gamma_y}) 
}
$$
for any ample line bundle $A$ and any $\pi$-exceptional effective divisor $F$ on $\Gamma$. 
The horizontal maps below is injective, and thus 
it is sufficient for the proof to find an ample line bundle $A$ on $X$  and 
a $\pi$-exceptional effective divisor $F$ on $\Gamma$ such that the right vertical map 
is surjective. 

We have the canonical bundle formula 
$$
K_{\Gamma}+\Delta'=\pi^\star (K_X+\Delta)+E', 
$$
where $\Delta'$ and $E'$ are the effective divisors with no common components. 
The support of $\Delta'+E'$ is normal crossing since $\pi$ is a log resolution of $(X, \Delta)$, 
and further we have $\lfloor \Delta' \rfloor=0$ since $(X,\Delta)$ has at most  klt singularities. 
Then, for the $\pi$-exceptional effective divisor $F$ defined by $F:=\lceil E' \rceil$, 
we can easily to check that 
$$
F':=F -E' \text{ is an effective $\mathbb{Q}$-divisor and } \lfloor \Delta' + F'\rfloor=0. 
$$
On the other hand, by the canonical bundle formula, 
we have 
\begin{align*}
\pi^\star (mL+2A)+F&=K_{\Gamma}+\pi^\star (mL+2A)+F-K_{\Gamma}\\
&=K_{\Gamma}+\pi^\star ((m+1)L+A)+ (\pi^\star A + \Delta' + F'). 
\end{align*}

Now we construct a quasi-psh function $\bar \rho$ on $\Gamma$ and 
a smooth hermitian metric $h $ on $\pi^\star A $ with the following properties\,$:$ 
\begin{itemize}
\item[$\bullet$] $\sqrt{-1}\Theta_{h }(\pi^\star A) + (1+\delta) \ddbar \bar \rho \geq 0$ holds 
for any $0<\delta \ll 1$.
\item[$\bullet$] $\bar \rho$ can be written as $m \log (\sum_{i=1}^m |z_{i}|^2)$ 
modulo addition of smooth functions on a neighborhood of a point in $\Gamma_y$, 
where $(z_{1}, z_{2},\dots, z_{n})$ is a local coordinate of $\Gamma$ 
such that $\Gamma_y=\{z_1=z_2=\cdots =z_m=0\}$. 
Here $n:=\dim X$ and $m:=\dim Y$. 
\end{itemize}
If $h $ and $\bar \rho$ with the above properties can be constructed, 
the singular hermitian metric 
$$ 
\text{
$H:=g_m  h  h_{\Delta' + F'}$ on the line bundle $\pi^\star ((m+1)L+A)+ (\pi^\star A + \Delta' + F')$ 
}
$$
satisfies that 
$$
\sqrt{-1}\Theta_{H}(\pi^\star ((m+1)L+A)+ (\pi^\star A + \Delta' + F')) + (1+\delta) \ddbar \bar \rho \geq 0
$$
for any $0<\delta \ll 1$. 
Here $g_m$ is a smooth hermitian metric (obtained from the nefness of $L$) 
on the line bundle $\pi^\star ( (m+1)L+A)$ with semi-positive curvature and 
$h_{\Delta' + F'}$ is the singular hermitian metric induced 
by the effective divisor $\Delta' + F'$. 
Then, by using the extension theorem in \cite{CDM17}, 
it can be shown that the induced map 
$$
H^{0}(\Gamma, \mathcal{O}_{\Gamma}(\pi^\star (mL+2A)+F)\otimes \I{H})
\to 
H^{0}(\Gamma, \mathcal{O}_{\Gamma}(\pi^\star (mL+2A)+F)\otimes \I{H}/
\I{ H e^{-\bar \rho} } )
$$
is surjective. 
Further, by the condition on the singularities of $\bar \rho$, 
we can see that 
$$
\I{H}=\mathcal{O}_{\Gamma} \quad \text{ and } \quad 
\I{He^{-\bar \rho}}=\mathcal{I}_{\Gamma_y} \text{ on a neighborhood of } \Gamma_y
$$
since $\Delta' + F'$ is a normal crossing $\mathbb{Q}$-divisor with $\lfloor \Delta' + F' \rfloor = 0$ and 
a general fiber $\Gamma_y$ transversely intersects with $\Delta' + F'$. 
This leads to the desired conclusion.

For the construction of $h $ and $\bar \rho$, 
we consider a very ample line bundle $A_Y$ on $Y$ and sections $\{s_i \}_{i=1}^{m}$ of $A_Y$ 
such that the common zero locus of $\{s_{i}\}_{i=1}^m$ is the disjoint union of 
$N$ points $\{y_i\}_{i=1}^N$
whose fibers do not intersect with the indeterminacy locus $B$ of $\psi$. 
Here $N$ is the self-intersection number of $A_{Y}$. 
The disjoint union of the fibers $F_i:=\psi^{-1}(y_i)$ is contained in 
the zero locus of the section $t_i:=\pi^\star s_i$ of 
the invertible sheaf $\pi^\star A_Y:=\pi_{\star}{\varphi}^\star A_Y$.

For a smooth hermitian metric $g$ on $\pi^\star A_Y$, the quasi-psh function $\rho$ on $X$ defined by  
$$
\rho:=m \log \sum_{i=1}^{m} |t_{i}|_g^2
$$
satisfies that $(1+\delta) \ddbar \rho \geq -C \sqrt{-1}\Theta_{g}$ for some constant $C>0$. 
Therefore we may assume that $A$ admits a smooth hermitian metric $g_A$ 
such that 
$$ 
\sqrt{-1}\Theta_{g_A}(A) + (1+\delta) \ddbar \rho \geq 0 
$$ by replacing $A$ with a positive multiple of $A$ (if necessarily) 
since $A$ is an ample line bundle on $X$. 
Strictly speaking, 
we need to consider hermitian metrics on invertible sheaves on complex spaces 
since $X$ may not be smooth. 
However it is not difficult to check that the above argument still works 
even when $X$ has singularities. 
Then it is easy to see that 
$h:=\pi^\star  g_A$ and $\bar \rho:=\pi^\star  \rho$ satisfy the desired properties. 
\end{proof}


\smallskip
It remains to prove that the converse inequality 
\begin{equation}\label{otherside}
\nd (-(K_X+\Delta)) \leq \nd (-(K_{X_y}+\Delta_{_{X_y}})).
\end{equation}
holds. 
This inequality follows from Proposition \ref{prop-rank} and 
the certain flatness of the direct image sheaves in Section \ref{Sec3}. 
In the rest of this subsection, 
we consider an MRC fibration $\psi: X \dashrightarrow Y$ 
under setting (\ref{setting}) and 
we use freely the notations in Section \ref{Sec3}. 
We put $D:=\Delta'$ and $E'':= E'-\varphi^{\star}N$
by following Remark \ref{rem-set}. 
Recall that $N$ is an effective $\mathbb{Q}$-divisor such that $N \sim_\mathbb{Q} K_Y$ 
and the support of $N$ is contained in $E$.

By the the canonical bundle formula in Remark \ref{rem-set} and 
the definitions of $\cV_{m,1} $ and $\widetilde{A}$ 
(see Proposition \ref{prop-fun3} and Proposition \ref{prop-fun5}), 
we have 
\begin{align*}
\cV_{m,1} &= \varphi_\star (-m (K_{\Gamma/Y} +\Delta') + c_{m,1} (E'-\varphi^{\star}N) + \widetilde{A} )\\
&=\varphi_\star (-m \pi^\star(K_{X} +\Delta) + m N +(c_{m,1} -m) E' + A + m_0 E)- 
\frac{1}{r_{m_0}}\det \varphi_\star (A+m_0 E). 
\end{align*}
Then, by the definition of the numerical dimension, we can see that 
\begin{align*}
&h^{0}(Y, \cV_{m,1} \otimes \frac{1}{r_{m_0}}\det \varphi_\star (A+m_0 E))\\
=&h^{0}(X, -m \pi^\star(K_{X} +\Delta) + m N +(c_m -m) E' + A+m_0 E)\\
\approx & O(m^{\nd (-(K_X+\Delta))}). 
\end{align*}
for $c_{m,1} \gg m >0$. 
Here we used that fact that $N$, $E'$, and $E$ are $\pi$-exceptional divisors. 
On the other hand, we can see that the sheaf $\varphi^{\star}\cV_{m,1}$ 
satisfies the assumptions in Proposition \ref{prop-rank} 
by Proposition \ref{prop-fun5}. 
Hence, from Proposition \ref{prop-rank}, we can obtain that  
\begin{align*}
h^{0}(Y, \cV_{m,1} \otimes \frac{1}{r_{m_0}}\det \varphi_\star (A+m_0 E))&=
h^{0}(\Gamma, \varphi^\star \cV_{m,1} \otimes \frac{1}{r_{m_0}} \varphi^\star\det \varphi_\star (A+m_0 E))\\
&\leq D \cdot \rank \cV_{m,1} \\
&\approx  O(m^{\nd (-(K_ {X_y}+\Delta_{X_y}))}). 
\end{align*}
Thus we obtain the desired inequality. 
\end{proof}

\subsection{Proof of Theorem \ref{splitgener}}
In this subsection, 
we prove the following result, which can be seen as a generalization of \cite{CH}.

\begin{theo}\label{splitgener}
Under setting $($\ref{setting}$)$, 
we further assume that $\psi: X \to Y$ is an MRC fibration 
between smooth projective varieties. 
Then we have the splitting
$$
T_X = V_1 \oplus V_2 \qquad\text{ on } X,
$$
where $V_1$ and $V_2$ are subbundles of $T_X$, and $V_1$ coincides with the $\pi_\star (T_{\Gamma/Y})$.

In particular, thanks to \cite{Hor07}, 
we can choose an MRC fibration $\psi : X\rightarrow Y$ 
to be a locally trivial fibration, namely,
for every small open set $U\subset Y$, we have the isomorphism $p: \psi^{-1} (U)\cong  U\times F$ that commutes with the fibration:
$$
\xymatrix{
\psi^{-1} (U) \ar[rd]_\varphi  \ar[rr]^p  && U \times F \ar[ld]
\\
& U &}
$$
where $F$ is the general fiber of $\psi$.
\end{theo}

\begin{rem}
The main difference between the above theorem and \cite{CH} is that 
$X$ is not necessary simply connected here. 
Following the arguments in \cite{CH}, we have already constructed 
the sheaves $\cV_{m, p}$ on $Y$ in Proposition \ref{prop-fun5}, 
which are weakly positively curved and satisfy that 
$\pi_\star (\varphi^\star c_1 (\cV_{m, p}))= 0$.
If $X$ is simply connected, we can thus easily prove that $\cV_{m, p}$ is a trivial vector bundle over $Y_0$, and thus $\varphi$ is  birational to a product. 
In the case $X$ is not simply connected, the argument here is more involved. 
\end{rem}

\begin{proof}
Let $\widetilde{A}$ be the $\varphi$-ample pseudo-effective line bundle on $\Gamma$ 
in Proposition \ref{prop-fun5}. We set 
$$
\cV_p := \varphi_\star (c_{0,p} E'' + p \widetilde{A}).
$$
Thanks to Proposition \ref{prop-fun5}, we know that $\pi_\star \varphi^\star \cV_p$ is weakly positively curved and $c_1 (\pi_\star \varphi^\star \mathcal{V}_p)=0$.
By using Proposition \ref{propnumflat}, 
its reflexive hull $(\pi_\star \varphi^\star \mathcal{V}_p)^{\star\star}$ is numerically flat vector bundle on $X$. In particular, it equips with a holomorphic flat connection $D$.

\smallskip

We check that the flat connection $D$ induces a flat connection on $\cV_p$ over $Y_0$ (after quitting a subvariety of codimension at least $2$, 
we can suppose that $\cV_p$ is locally free on $Y_0$).
We fix a hermitian metric $h$ on $\cV_p |_{Y_0}$ and 
let $H$ be the hermitian metric on $\varphi^\star \mathcal{V}_p$ defined by the pull-back $\varphi^{\star} h$.
Let $D_{H}$ be the Chern connection associated to 
the hermitian vector bundle $(\varphi^\star \mathcal{V}_p |_{X_0}, H)$, 
where $X_0:=\varphi^{-1}(Y_0)$.
For a local frame $\{f_i\}_{i=1}^r$ of $\cV_p$ on an open set $U \subset Y_0$, 
the pull-back $\{\varphi^{\star}f_i\}_{i=1}^r$ gives a local trivialization of $\varphi^{\star} \cV_p$ 
on $\varphi^{-1}(U) \subset X_0$. 
Let 
$$
\Gamma_H \text{ and } 
\Gamma \in C^{\infty} _{1,0} (\varphi^{-1}(U), \eend (\varphi^\star \cV_p )) 
$$ 
be the connection form of $D_H$ (resp. $D$)  
on $\varphi^{-1}(U) \subset X_0$. 
Then we have  
$$
D = D_H + \Gamma- \Gamma_H.  
$$
It follows that $\Gamma_H$ is obtained from the pull-back of the connection form of $h$ 
from $H=\varphi^{\star}h$. 
Hence $\Gamma$ can be regarded as a matrix of the $(1,0)$-form $\Gamma_{ij}$ 
defined on $\varphi^{-1}(U) $ by the fixed trivialization. 
Hence it is sufficient to show that $\Gamma_{ij}$ is also obtained 
from the pull-back. 
It follows that 
\begin{equation}\label{flatcc}
D \Gamma + \Gamma \wedge \Gamma=0.  
\end{equation}
since $D$ is a flat connection. 
The $(1,1)$-part of the left hand side is $\dbar \Gamma$ 
since $\Gamma$ is of $(1,0)$-type. 
This implies that $\Gamma_{ij}$ is a holomorphic $1$-form defined on $\varphi^{-1}(U)$. 

A general fiber $X_y$ has no holomorphic forms as it is a rationally connected manifold, 
and thus the restriction $\Gamma_{ij}|_{X_y}$ should be identically zero. 
Hence there exists a Zariski open set $Y_1 \subset Y_0$ 
and holomorphic $1$-forms $\eta_{ij}$ on $Y_1$
such that $\Gamma_{ij}=\varphi^{\star} \eta_{ij}$ holds on $Y_1$. 
Further, by taking a smaller Zariski open set $Y_1$ (if necessarily), 
we may assume that $\varphi: X_1:=\varphi^{-1}(Y_1) \to Y_1$ is 
a smooth morphism with rationally connected fibers. 

We now show that $\eta_{ij}$ can be extended to the holomorphic $1$-form on $Y_0$. 
We first take a open neighborhood $U$ of an arbitrary point in $Y_0 \setminus Y_1$ 
and a local coordinate $(z_1, z_2, \dots, z_m)$ on $U$, where $m:=\dim Y$. 
The holomorphic $1$-form $\eta_{ij}$ on $Y_1$ can be locally written as 
$$
\eta_{ij}=\sum_{k=1}^{m} f_k dz_k, 
$$
where $f_k$ is a holomorphic function on $U \cap Y_1$. 
To extend $f_k$ to the holomorphic function on $U$, 
we consider the $m$-form 
$$
F:=f_k dz_1\wedge dz_2 \cdots  \wedge dz_m. 
$$
By Fubini's theorem, 
we obtain 
\begin{align*}
\int_{y \in U \cap Y_1}  c_m\,  F \wedge \overline{F}  \int_{X_y} \omega^{n-m} 
=  \int_{\varphi^{-1}(U \cap Y_1)} \varphi^{\star}F \wedge \overline{\varphi^{\star}F} \wedge \omega^{n-m}, 
\end{align*}
where $c_m:=\sqrt{-1}^{m^2}$ and $\omega$ is a K\"ahler form on $X$. 
Since $\Gamma_{ij}=\varphi^\star \eta_{ij}$ is a holomorphic $1$-form on $X_0$,
the integrand $\varphi^\star F$ in the right hand side can be extended to $X_0$, 
and thus the right hand side converges. 
The fiber integral of $ \omega^{n-m}$ over $X_y$ does not depend on $y \in U \cap Y_1$.
Hence $F$ is $L^2$-integrable, and thus it can be extended by the Riemann extension theorem. 
For  the Chern connection $D_h$ of $(\cV_p |_{Y_0}, h) $, 
it is easy to check that 
$D_{\rm{flat}} := D_h + (\eta_{i,j})$ defines a flat connection of $\cV_p$ over $Y_0$. 
\medskip
 
Now we are in the same situation as in \cite[3.C]{CH}, namely $\widetilde{A}$ is a $\varphi$-ample line bundle and $\cV_p = \varphi_\star (c_{0,p} E' + p \widetilde{A})$ 
is flat on $Y_0$ for every $p\in\mathbb{Z}_+$. Following the same argument as in \cite[3.C]{CH}, the flat connection 
$ (\cV_p, D_{\rm{flat}})$ over $Y_0$ induces a birational morphism from $\varphi: X_0 \rightarrow Y_0$ 
to a locally  trivial fibration, and the birational morphism does not contract any non $\pi$-exceptional divisor in $X_0$.
It induces thus a splitting:
$$T_X = V_1 \oplus V_2 \qquad\text{on } X\setminus S,$$
for some subvariety $S$ of codimension at least $2$ in $X$. 
As $X$ is smooth, the splitting can be extended on the total space and the theorem is proved.
\end{proof}
\bigskip

\subsection{Proof of Theorem \ref{thm-maincam}}
\label{subsec-4.3}
Finally we prove Theorem \ref{thm-maincam}.

\begin{proof}[Proof of Theorem  \ref{thm-maincam}]
Thanks to Theorem \ref{splitgener}, we can take an MRC fibration $\psi : X\rightarrow Y$ such that it is locally trivial fibration and $c_1 (Y)=0$. 
We need to prove that the identification $p: \psi^{-1}(U) \cong  U \times F$ in Theorem \ref{splitgener} identifies also the pair $\Delta$. I
\smallskip

Let $\Delta = \Delta^{\hor} +\Delta^{\verti}$, where $\Delta^{\verti}$ is the vertical part 
(that is, $\psi_\star (\Delta^{\verti}) \subsetneq Y$) and $\Delta^{\hor}$ is the horizontal part.
We first prove that $\Delta^{\verti} = 0$. 
Indeed, we have
$$
-\Delta^{\verti} -K_Y = K_{X/Y} - (K_X +\Delta) + \Delta^{\hor}.
$$
Note that the restriction of $K_{X/Y} - (K_X +\Delta) + \Delta^{\hor} $ 
on the general fiber of $\psi$ is effective. 
By applying \cite{BP08}, we can see that 
$K_{X/Y} - (K_X +\Delta) + \Delta^{\hor}$ is pseudo-effective. 
As a consequence, the divisor $-\Delta^{\verti} -K_Y$ is pseudo-effective.
Hence we obtain $\Delta^{\verti} = 0$ by $c_1 (Y)=0$.

Let $\widetilde{A}$ be the line bundle constructed in Proposition \ref{prop-fun5}. 
We recall that $\widetilde{A}$ is pseudo-effective and 
$\psi_\star (q \widetilde{A})$ is numerically flat for every $q\in \mathbb{Z}_+$.  
Now we prove that 
$\psi_\star (p\Delta + q \widetilde{A})$ is a numerically flat vector bundle on $Y$ 
for every  $p, q\in \mathbb{Z}_+$.



By Lemma \ref{withpairlem}, 
we know that  $\psi_\star (p\Delta + q \widetilde{A} )$ is weakly positively curved. 
Hence $\theta$ defined by 
$$
\theta := c_1 \left(\psi_\star (p\Delta +q \widetilde{A}) \right)
$$
is also pseudo-effective. 
It is sufficient to show that $\theta=0$ by Proposition \ref{propnumflat}. 
It follows that 
$$
L := p\Delta +q \widetilde{A} -\frac{1}{r} \psi^\star\theta
$$ 
is pseudo-effective 
from Lemma \ref{prop-fun2} and pseudo-effectivity of $p\Delta +q \widetilde{A}$. 
Here $r$ is the rank of $\psi_\star (p\Delta + q \widetilde{A} )$. 
Further $L$ is also $\psi$-relatively big by the definition. 
Therefore, by applying Lemma \ref{thm-cur} to $m \gg p$ and the following form:   
$$
m K_{X/Y} -m (K_{X/Y} +\Delta) +(m-p)\Delta + L =  q \widetilde{A} -\frac{1}{r} \psi^\star\theta,
$$
we can see that 
$\psi_\star (q \widetilde{A} -(1/r) f^\star\theta )$ is weakly positively curved.
This implies that 
$$
c_1 (\psi_\star (q \widetilde{A})) \geq c \theta
$$ for some constant $c>0$. 
On the other hand, by Proposition \ref{prop-fun5}, 
we have already known that $c_1 (\psi_\star (q \widetilde{A}))=0$, 
which implies that $\theta =0$. 
Then, by applying Proposition \ref{isotri}, the theorem is proved.
\end{proof}

\section{sRC quotients for orbifolds with nef anti-canonical divisors}\label{sRC}

\subsection{Orbifold surfaces with nef anti-canonical divisors}
The aim of this section is to prove Conjecture \ref{mainconsRC} for projective surfaces. 
We remark that Theorem \ref{thm-sRC-intro} can be derived from Theorem \ref{thm-slope} and Theorem \ref{product} below. 

\begin{theo}\label{thm-slope}
Let $(S, D)$ be a klt pair such that $S$ is a smooth surface, 
the support of $D$ is normal crossing, and $- (K_S +D)$ is nef. 
Then there exists a fibration $\rho : (S, D)\rightarrow (R, D_R)$ onto the orbifold base $(R, D_R)$
with the following properties$:$
\begin{enumerate}
\item $(R, D_R)$ is a klt pair such that $R$ is smooth and $c_1 (K_R + D_R)=0$.
\item General orbifold fibers are sRC. 
\item The fibration is locally trivial with respect to pairs, namely, for any small open set $U \subset R$, we have the isomorphism
$$(\rho^{-1} (U), D)\cong  (U, D_R |_U) \times (S_r, D |_{S_r})$$
over $U \subset R$. Here $S_r$ is a general fiber of $\rho$. 
\end{enumerate}
In particular, Conjecture \ref{mainconsRC} is true in the two dimensional cases. 
\end{theo}

\begin{rem}
Example \ref{exlc} at the end of this subsection  shows that the assertion of Conjecture \ref{mainconsRC} 
fails for general (non-klt) lc  pairs. 
\end{rem}

\begin{proof}
We first consider the case of $h^1(S,\mathcal{O}_S) \neq 0$. 
In this case, the surface $S$ is not rationally connected. 
Hence, by applying Theorem \ref{thm-maincam}, we can find a a locally trivial MRC fibration 
$\psi: S \rightarrow Y$ onto a smooth manifold $Y$ satisfying that $c_{1}(Y)=0$ and $\dim Y >0$. 

In the case of $\dim Y =2$ (that is, $\psi$ is the identity map of $S$), 
the surface $S$ is not uniruled, and thus the canonical bundle $K_S$ is pseudo-effective by \cite{BDPP} and \cite{GHS03}
Hence we can see that 
$$-D=- (K_S +D) + K_S$$ 
is pseudo-effective. 
This implies that $D =0$. 
Then the sRC fibration of $S$ is also the identity map in this case.

Now we consider the case of $\dim Y =1$. 
We remark that all the fibers are $\mathbb{P}^1$ in this case. 
If $-\deg (K_S + D)|_F>0$, then all the fibers are sRC. 
Hence, by setting $D_Y:=0$, 
the MRC fibration $\psi$ itself coincides with the sRC fibration and it  satisfies all the conditions of the theorem. 
If $-\deg (K_S + D)|_F=0$, then $-(K_S + D)$ is numerically equivalent to $aF$. 
It follows that $a \geq 0$ since $-(K_S + D)$ is nef. 
In the case of $a>0$, the pair $(S, D+\delta F)$ is still a klt pair with nef anti-log canonical divisor 
for any small $\delta>0$. 
Then, by Theorem \ref{thm-maincam}, 
the divisor $D+\delta F$ should be horizontal with respect to the MRC fibration, 
which is a contradiction.
In the case of $a=0$, the divisor $-(K_S + D)$ is numerically zero, 
and thus the identity map of $S$ gives the sRC fibration 
satisfying all the conditions of the theorem 
by setting $R:=S$ and $D_Y:=D$.

For the remaining case of $h^1(S,\mathcal{O}_S)  = 0$, 
we can obtain the more precise decomposition theorem (Theorem \ref{thm-slope}). 
Theorem \ref{thm-slope} finishes the proof of Theorem \ref{thm-slope}. 
\end{proof}

The rest of this subsection is devoted to the proof of the following theorem: 

\begin{theo}\label{product} 
Let $(S,D)$ be a klt pair such that $S$ is a compact K\"{a}hler surface,  
the support of $D$ is normal crossing, and $-(K_S+D)$ is nef. 
We further assume that $h^1(S,\mathcal{O}_S)=0$ and $(S,D)$ is not sRC. 
Then the pair $(S,D)$ is a product with respect to pairs: 
$$
(S,D)=(B,D_B)\times (\mathbb{P}^1,D_{\mathbb{P}^1}), 
$$
that is, $S=B\times \mathbb{P}^1$ and $D=p^\star(D_B) + q^{\star}( D_{\mathbb{P}^1})$.
\end{theo}

\begin{proof}
By applying \cite[Theorem 1.5]{Cam16} (see also Theorem \ref{thm-cam}) to $(S,D)$, we obtain a birational map 
$g:(S',D')\to (S,D)$ with the following properties:
\begin{enumerate}
\item[$\bullet$] $g_{\star}(D')=D$. 
\item[$\bullet$] $S'$ is smooth and $(S',D'=\overline{D}+\Delta')$ is klt. 
\item[$\bullet$] $g^{\star}(K_S+D)=K_{S'}+\overline{D}+\Delta'-E^{\bullet}$. 
\end{enumerate}
Here $\Delta'$ and $E^{\bullet}$ are effective $g$-exceptional divisors without common component, and 
$\overline{D}$ is the strict transform of $D$. 
We moreover can take a fibration $f: (S',D')\to (B,D_B)$ such that 
$(B,D_B)$ is the orbifold base of $(S',D')$. 
We have the following commutative diagram:  
$$
\xymatrix{
(S', D') \ar[rd]_{f} \ar@{^>}[rr]^g & &  (S,D) \ar@{-->}[ld]^{f\circ g^{-1}}\\
& (B, D_B)
}
$$
On the other hand, thanks to \cite[Corollary 10.4]{Cam16} and the assumption of nefness, 
we obtain 
\begin{align}\label{eq-zero}
\nd (K_B+D_B)=\kappa(K_B+D_B)=0.
\end{align}
Then there are two possibilities: 
either $B=\mathbb{P}^1$ with $\degree (D_B)=2$, or 
$B$ is an elliptic curve with $D_B=0$.  This second case is excluded by our assumption that $h^1(S,\mathcal{O}_S)=0$.
The proof of the theorem consists of the following steps:

\medskip

\begin{enumerate}
\item[Step 1] By Lemma \ref{l1}, we show that the map 
$f\circ g^{-1}:S \dashrightarrow B$ is actually a holomorphic map. 
We may thus  assume that $g=id_S$, $S'=S$, and $D'=D$. 
Then, since $S \to B$ is a (not necessarily minimal) ruled surface, 
it can be factorized into a sequence of blow-downs $g:S\to P$ and  a $\mathbb{P}^1$-bundle $h:P\to B$. 
Here  $P$ is a smooth surface and $f=h\circ g$ holds  (see the following diagram). 
We define the divisor $D_P$ on $P$ by $D_P:=g_{\star}(D)$. 
$$
\xymatrix{
 (S,D) \ar@{>}[rd]_{f}  \ar@{>}[rr]^{g} & &   (P, D_P) \ar@{>}[ld]^{h}\\
& (B, D_B)
}
$$

\item[Step 2] By Lemma \ref{l2}, we show that $-(K_P+D_P)$ is nef. 
\item[Step 3] By Lemma \ref{l3}, we show that $(P, D_P) $ is a product of $(B,D_B)$ and  
$(\mathbb{P}^1,D_{\mathbb{P}^1})$,  as in the statement of Theorem \ref{product}. 
Here $D_{\mathbb{P}^1}$ is the restriction of $D_P$ to a general fiber of $h:P\to B$. 

\item[Step 4] By Lemma \ref{lproduct}, we show that $g=id_P$, $S=P$, and $D=D_P$.
\end{enumerate}
\end{proof}

We now state the four lemmas and prove them with the same notation as above. 

\begin{lemm}\label{l1} In the same situation as in Theorem \ref{product}, 
the map $f\circ g^{-1}$ is regular. 
\end{lemm}

\begin{proof} 
We assume that $f\circ g^{-1}$ is not regular. 
Then we can choose  a $g$-exceptional component $E\cong \mathbb{P}^1$ of either $\Delta'$ or $E^{\bullet}$ such that $f(E)=B$. 
Recall that the support of $\overline{D}+\Delta'+E^{\bullet}$ is 
snc and that the coefficient of each component of $D'$ lies in $[0,1)$.
The divisor $E$, being $g$-exceptional, 
meets other components of $D'$ at most $2$ and each of them has coefficient in $[0,1)$. 
Hence, in the first case (that is, the case of $E \subset E^{\bullet})$, 
we can see that $\degree (K_E+D'|_E)<0$. 
In the second case (that is, the case of $E \subset \Delta'$), 
we put $D'':=D'-cE$, where $c\in [0,1)$ is the coefficient of $E$ in $D'$. 
Then we have $\degree (K_E+D''|_E)<0$ by the same reason.


Let us now consider the case of  $E\subset \Delta'$. 
Now $B$ is a curve, and thus there are no $f$-exceptional divisors on $S'$. 
Further $(B,D_B)$ is the orbifold base of $(S',D')$ and $E$ is $f$-horizontal. 
Hence we can see that 
the natural maps $(S',D')\to (B,D_B)$ and $(S,D'')\to (B,D_B)$ are orbifold morphisms with the same orbifold base. 
By considering the restriction of $f$ to $E$, 
we obtain the orbifold morphism $m:(E,D'')\to (B,D_{B,E})$ with the orbifold divisor $D_{B,E}$.   
The divisor $E$ is $f$-horizontal and all the components of $D''$ transversally intersect with $E$. 
Further  we can take a fiber $F_b:=f^{-1}(b)$ of $b \in B$ that intersects with $E$, 
and  thus we have $t_{F_b} m_{F_b,D'}\geq m_{b, D_B}$, 
where  $t_{F_b}$ (resp. $m_{F_b,D'}$, $m_{F_b,D_B}$) is the multiplicity of $F$ in $f^{\star}(b)$ 
(resp. $D'$, $D_B$).
It follows that $D_{B,E}\geq D_{ B}$ from the above inequality.

Since $E$ and $B$ have the same dimension (one), we get 
$$m^{\star}(K_B+D_{B,E}) \leq  (K_E+D''|_E)$$ 
and thus we obtain
$$
0=\degree(m^{\star}(K_B+D_B))\leq \degree(m^{\star}(K_B+D_{B,E}))\leq \degree (K_E+D''|_E)<0. 
$$ 
This is a contradiction.

For the remaining case (that is, $E \subset E^{\bullet}$),  
the same argument, but applied to $K_E+D'|_ E$ (instead of $K_E+D''|_ E$),  works. 
Thus the proof of the lemma finishes.    
\end{proof}




\begin{lemm}\label{l2} 
Let $(S,D)$ be an arbitrary orbifold pair such that $S$ is a smooth surface 
and let $g:S\to T$ be the contraction of a $(-1)$-curve of $S$. 
Set $D_T:=g_{\star}(D)$. 
If $-(K_S+D)$ is nef, then so is $-(K_T+D_T)$. 
In particular, the divisor $-(K_P+D_P)$ is nef. 
\end{lemm}

\begin{proof} 
We have $K_S+D=g^{\star}(K_T+D_T)+\ell E$ for some $\ell\in \mathbb{Q}$. 
It follows that $\ell$ is non-negative since we have 
$$
0\leq -(K_S+D)\cdot E=-\ell \cdot E^2=\ell. 
$$
Therefore it can be seen that $-g^{\star}(K_T+D_T)=-(K_S+D_S)+\ell E$  is nef by the projection formula, 
which leads to the lemma. 
\end{proof}

\begin{rem} 

Even in the case of $h^1(S,\mathcal{O}_S) \not =0$, 
Lemmas \ref{l1} is true and its proof is  much easier. 
Further Lemma \ref{l1} and Lemma \ref{l2} are true in more general situations: 
Lemma \ref{l1}  holds for any klt surface pairs and Lemma \ref{l2} holds 
for any $K$-negative contractions by the \lq \lq negativity lemma". 
\end{rem}

\begin{lemm}\label{l3} 
Let $h:P\to B$ be a $\mathbb{P}^1$-bundle over  a smooth curve $B=\mathbb{P}^1$. 
Let $D_P$ be a $($not necessarily snc supported$)$ klt orbifold divisor on $P$ with $-(K_P+D_P)$ nef 
and  $D_B$ be the divisor on $B$ that give the orbifold base of $(h,D_P)$. 
Then $(P,D_P)$ is isomorphic to the product orbifold $(B,D_B)\times(\mathbb{P}^1,D_{\mathbb{P}^1})$.
\end{lemm}


\begin{proof} 
We may assume that $P$ is the Hirzebruch surface $\mathbb{F}_m$ containing a section $C \subset P=\mathbb{F}_m$ 
satisfying that  $C^2=-m$, where $m\geq 0$. 
The divisor $D_P$ can be decomposed into the vertical part $D_P^{\verti}$ and the horizontal part $D_P^{\hor}$. 
Then we have 
$$
D_P^{\verti}=p^{\star}(D_B) \quad \text{ and } \quad  D_P^{\hor}=a C+\sum_j a_j H_j, 
$$
where $a$ and $a_j$ are positive rational numbers, and $H_j$ are distinct irreducible curves.
It follows that 
$H_j $ is numerically equivalent to $d_j(C+b_j F)$ for some $b_j\geq m$ from $H_j \cdot C\geq 0$, 
which we denote by 
$$
H_j \equiv d_j(C+b_j F), 
$$ 
where $F$ is a general fiber of $h$  and 
$d_j$ is the positive integer defined by the intersection number $d_j:=H_j\cdot F$. 
Then we have 
$$
-(K_P+D_P)\equiv 2C+(m+2)F-2F-(a C+\sum_ja_j d_j (C+b_j F)).
$$
By computing the intersection of $-(K_P+D_P)$ with $C$, 
we get 
\begin{align*}
0\leq -(K_P+D_P) \cdot C &=-m+(a+\sum_ja_j d_j)m -\sum_j a_j b_j d_j \\
&= -m+a m +\sum_j a_j d_j \cdot (m- b_j),
\end{align*}
This implies that $m=0$ by $b_j\geq m$ and $a<1$ (which comes from by the klt condition). 
Therefore we can conclude that $B=\mathbb{F}_0=\mathbb{P}^1\times \mathbb{P}^1$. 
Further we can see that $b_j=0$, and thus $H_j \equiv C$ for all $j$.
\end{proof}

\begin{rem} 
In the case of $B$ being an elliptic curve, the preceding computation shows that 
$(P,D_P)$ is either a product, or $P$ is the \lq \lq non-decomposable" ruled surface over $B$ 
with minimal self-intersection curve $C$ such that $C^2=0$, or $C^2=1$. 
Further, in the former case of $C^2=0$, we can see that $D_P=0$ by the same argument. 

In our situation, the canonical bundle $(K_B+D_B)$ is numerically trivial (see (\ref{eq-zero})), 
by the argument in the proof of Theorem \ref{product}. 
Hence $B=\mathbb{P}^1$ with $\degree (D_B)=2$, or 
$B$ is an elliptic curve with $D_B=0$.  
The latter case contradicts to the assumption of $0=h^{1}(S, \mathcal{O}_S)=h^{1}(P, \mathcal{O}_P)$. 
Hence it is sufficient to consider the former case of $B=\mathbb{P}^1$ for the proof of Theorem \ref{product}. 
\end{rem}

\begin{lemm}\label{lproduct} 
Let $(P,D_P)$ be a product orbifold $(B,D_B)\times(\mathbb{P}^1,D_{\mathbb{P}^1})$, 
where the both factors are klt orbifold curves with $\deg (K_B+D_B)=0$. 
Let  $(S,D)$ be a klt pair such that $S$ is smooth and $-(K_S+D)$ is nef. 
If $g:(S,D)\to (P,D_P)$ is a birational morphism satisfying that $D_P=g_{\star}(D)$, 
then we have $S=P$, $g=id_P$, and $D=D_P$.
\end{lemm}

\begin{proof} 
The birational morphism $g$ is the composition of blow-ups of points, 
and thus, by Lemma \ref{l2},  it is sufficient to prove the lemma when $g$ is the blow-up $g:S\to P=B\times F$ along one point $r\in P$. 
Let $E$ be the the exceptional divisor of $g:S\to P=B\times F$ along $r \in P$. 
Let $F'$ (resp.$C'$) be  the strict transform in $S$ of the fiber $F$ (resp. $C$) of  $p$ 
(resp. $q$) passing through the blow-up center $r$, 
where $p$ (resp. $q$) is the first (resp. second) projection of the product $B\times F$. 
Further let $m_F \in [1, \infty]$ (resp. $m_C \in [1, \infty]$) be the multiplicity of $F$ (resp. $C$) in $D_P$. 

By the assumption, the divisor $D_P$ can be obtained from the sum of the pull-backs of $D_B$ and $D_{\mathbb{P}^1}$, 
and thus we have 
$$
D_P=\Gamma + \big(1-\frac{1}{m_F} \big)F+ \big(1-\frac{1}{m_C} \big)C, 
$$
where $\Gamma $ is the effective divisor whose support does not contain the point $r$. 
We remark that the case of $m_F=1$ (resp. $m_C=1$) corresponds to 
the case where $r$ does not lie in the vertical divisor $D_P^{\verti}$ (resp. in the horizontal divisor $D_P^{\hor}$). 
On the other hand, when we define $\delta$ by the intersection number $\delta:= D_P^{\hor} \cdot F$, 
we have 
$$
D_P=p^{\star}(D_B)+q^{\star} (D_{\mathbb{P}^1}) \equiv 2 F + \delta C. 
$$
by $\degree (K_B+D_B)=0$. 

We now compute the numerical class of $-(K_S+D)$ in $H^{1,1}(S, \mathbb{Q})$, 
which is generated over $\mathbb{Q}$ by the numerical  classes of $F'$, $C'$, and $ E$. 
When we define $\overline{D}$ by the strict transform of $D$, 
we have $K_S+D=K_S+\overline{D}+sE$ for some $1>s\geq 0$. 
A straightforward computation yields that 
\begin{align*}
K_S &= -2(F'+E)-2(C'+E)+E=-(2F'+2C'+3E),\\
\overline{D}&\equiv2(F'+E)+\delta(C'+E)-[(1-\frac{1}{m_F})E+(1-\frac{1}{m_C})E] 
= 2F' +\delta C' +  (\delta + \mu)  E, 
\end{align*}
where $\mu:=(\frac{1}{m_F}+\frac{1}{m_C})$. 
Hence we obtain that 
$$
-(K_S+D)\equiv (2-\delta)C'+\big(3-(\delta+ \mu+s) \big)E.
$$
By using the equalities of the intersection numbers $C'^2=E^2=-1$ and  $C'\cdot E=1$,
we can see that 
$$
-(K_S+D) \cdot E=(K_S+D) \cdot C'=-1 + (\mu + s). 
$$
Hence it follows that $1=\mu+s$ from the assumption where $-(K_S+D)$ is nef. 

We denote by $t$ the multiplicity of $E$ in $D$ (that is, $1-\frac{1}{t}=s$). 
Then we have $t\geq m_F$ since $(B,D_B)$ is the orbifold base of $p\circ g: (S,D)\to B$. 
Together with the equality $s+\mu=1$, 
this gives 
$$
\frac{1}{m_F} \geq \frac{1}{t}=\frac{1}{m_F}+\frac{1}{m_C}. 
$$
Hence we can see that $t=m_F $ and $1/m_C=0$, 
which contradicts the klt property of $(S,D)$.

\end{proof}

\begin{ex}\label{exlc}
We construct a smooth lc  orbifold surface $(S,D)$ such that the anti-canonical divisor $-(K_S+D)$ is nef  and 
$S$ is rationally connected, but its sRC fibration does not decompose $S$ as a product. 

Start with $S_0=\mathbb{P}^1\times \mathbb{P}^1$ with projections $p$, $q$ onto its first and second factor. 
Let $D_0=V+V'+H$, where $V,V'$ are two (reduced) fibers of $p$ and $H$ is a fiber of $q$. 
Then we have $K_{S_0}+D_0=-H_0$. 
Let $g:S\to S_0$ be the blow-up of the two points of  the intersection points of $H$ with $V,V'$, 
and let $E$, $E'$ be the corresponding exceptional divisors. 
Let $H',W,W'$ be the strict transforms of $H,V,V'$ respectively and 
let $p':=p\circ g$, $q':=q\circ g: S\to \mathbb{P}^1$ be the composition morphisms. 
Then we have 
$$
-g^{\star}(H_0)=g^{\star}(K_{S_0}+D_0)=H'+E+E'=K_S+W+E+W'+E'+H. 
$$
If we define the divisor $D$ by $D:=W+E+W'+E'+H'$, 
then $-(K_S+D)$ is nef by the above formula. 
However the sRC fibration of $(S,D)$ coincides with $p':(S, D)\to (\mathbb{P}^1, D_{\mathbb{P}^1})$, 
where $D_{\mathbb{P}^1}$ is the orbifold divisor defined by the two points of the images of $V,V'$.
Then $p':(S, D)\to (\mathbb{P}^1, D_{\mathbb{P}^1})$ is obviously not a product.
\end{ex}

\subsection{Orbifolds with nef anti-canonical divisors}

To prove Conjecture \ref{mainconsRC} for higher dimensional cases, the main difficult is that, because of the appearance of the orbifold base $D_R$, 
we have not yet at the moment the semistability theorem like Theorem \ref{prop-fun}. 
However we still have some results in this direction.

\begin{theo}\label{tfsrc} 
Let $(X,D)$ be a smooth orbifold pair such that it is lc and and the  anti-canonical divisor $-(K_X+D)$ is nef. 
Let $f:(X',D')\to (Z,D_Z)$ and $g:(X',D')\to (X,D)$ be a smooth neat representative of its $sRC$ quotient 
with following commutative diagram$:$ 
$$
\xymatrix{
(X', D') \ar[rd]_{f} \ar@{^>}[rr]^g & &  (X,D) \ar@{-->}[ld]^{}\\
& (Z, D_Z)}
$$
Then we obtain $g_{\star}(f^{\star}(K_Z+D_Z)) =0$ and $g_{\star}(D(f,D))=0$, 
that is, these two effective divisors  are $g$-exceptional.
Here $D(f,D)$ is the divisor defined by $D(f,D):=D(f)-D(f)_{/D_Z}$ 
$($see the proof below for the definitions of $D(f)$ and $D(f)_{/D_Z}$$)$. 

\end{theo} 

\begin{proof} 
We first remark that $K_Z+D_Z$ is pseudo-effective by \cite[Theorem 1.5, Corollary 10.6]{Cam16} 
(see Theorem \ref{thm-cam}). 
We may assume that $g:(X',D')\to (X,D)$ be a birational morphism such that  $(X',D')$ is a log smooth  pair. 
We have the formula of the canonical bundles: 
$$
g^{\star}(K_X+D)=K_{X'}+D'-E^{\bullet} \text{ and } D'=D''+\Delta', 
$$
where  $D''$ is  the strict transform of $D$ in $X'$ and 
$\Delta'$, $E^{\bullet}$ are the effective $g$-exceptional divisors without common component. 
We may assume that $f:(X',D')\to Z$ is  \lq \lq neat" and its orbifold base $(Z,D_Z)$ is smooth. 
We remark that $(X',D')$ is also lc since we are assuming that $(X,D)$ is lc. 

For the divisor $N'$ defined by 
$$
-N':=g^{\star}(K_X+D)=K_{X'}+D' -E^{\bullet}, 
$$
we can see that $N'$ is nef from the assumption. 
Thus $N'+\beta A'$ is ample on $X'$ for an arbitrary small rational numbers $\beta>0$ and a fixed polarization $A'$  on $X'$. 
We choose an effective $\mathbb{Q}$-divisor $B'$ on $X'$ with the following properties: 
\begin{enumerate}
\item $B'$ is $\mathbb{Q}$-linearly equivalent to $N'+\beta \cdot A'$. 
\item $B'$ is $f$-horizontal. 
\item $D^+:=D'+B'=D'' + \Delta ' +B'$ has snc support. 
\item $(X',D^+)$ is still an lc pair. 
\end{enumerate}
Note that 
$$
K_{X'}+D^+=K_{X'}+D' + B'\equiv E^{\bullet}+\beta \cdot A'.
$$

Both $f:(X',D^+)\to Z$ and $f:(X',D')\to Z$ have the same orbifold base $(Z,D_Z)$ 
since $B'$ is $f$-horizontal. 
For the orbifold base $(Z,D_Z)$ of $f:(X',D^+)\to Z,$ 
by applying \cite[Theorem 4.11]{Cam04} and \cite[Theorem 3.4]{CP19} with a small adaptation, 
we can see that 
$$
P_{\beta} := K_{X'/Z}+(D^+)^{\hor}-D(f)- E'
$$
is pseudo-effective, where $E'$ is (not necessarily effective) $g$-exceptional. 
Here $D(f)$ (resp. $D(f)_{/D_Z})$ is the effective $f$-vertical divisor defined by 
$$ 
D(f)=\sum_{F\subset X'}(t_F-1) \cdot F, 
$$
where $F$ runs over all prime divisors of $X'$ such that the image $f(F)$ by $f$ has codimension one 
(resp. such that, moreover, $f(F)$ is contained in $D_Z$), and 
$t_F$ is the multiplicity of $F$ in $f^{\star}(f(F))$.

By the lemma below, we can see that 
$$
f^{\star}(D_Z)=D(f)_{/D_Z}+(D')^{\verti}-\Delta, 
$$ for some effective divisor $\Delta$, and thus, by $(D^+)^{\verti}=(D')^{\verti}$, we obtain 
\begin{align*}
P_{\beta}+E'&=K_{X'/Z}+(D^+)^{\hor}-D(f) \\
&=K_{X'}+D^+-f^{\star}(K_Z+D_Z)-[D(f)-D(f)_{/D_Z}]-\Delta \\
&=K_{X'}+D^+-f^{\star}(K_Z+D_Z)-D(f,D)-\Delta. 
\end{align*}
Here $D(f,D):=D(f)-D(f)_{/D_Z}$ is the effective divisor consisting of the components of $D(f)$ whose image by $f$ 
is not contained in $D_Z$. 
Thus we have 
$$
P_{\beta}+f^{\star}(K_Z+D_Z)\equiv E^{\bullet}-E'+\beta\cdot A'-D(f,D) -\Delta.
$$ 
Then, by letting $\beta\to 0^+$, 
we can see that the pseudo-effective class $P$, which is the limit of $P_{\beta}$, 
satisfies that 
$$
P+f^{\star}(K_Z+D_Z)=E^{\bullet}-E'-D(f,D)-\Delta. 
$$

Let $C\subset X$ be a general curve constructed by the complete intersection of ample divisors avoiding $g(E')$. 
We define the curve $C'$ on $X'$ by the inverse image of $C$  in $X'$, so that $C'$ does not intersect with $ E^{\bullet}$ and $E'$. 
Then we have 
$$
0\leq C'\cdot f^{\star}(K_Z+D_Z)\leq C'\cdot (P+f^{\star}(K_Z+D_Z))=-C'\cdot (D(f,D)+\Delta)\leq 0.
$$ 
Note that the first inequality comes from the pseudo-effectivity of $K_Z+D_Z$. 
The curve $C'$ satisfies that  $C'\cdot M>0$ for any non-$g$-exceptional effective divisor $M$ on $X'$, 
and thus we obtain the conclusion by $C'\cdot f^{\star}(K_Z+D_Z)= C'\cdot (D(f,D))=0$. 

\end{proof}

 \begin{lemm} \label{lcam}
 With the above notations, there exists an effective divisor $\Delta$ on $X'$ such that  
 $f^{\star}(D_Z)=D(f)_{/D_Z}+(D')^{\verti}-\Delta$.
 \end{lemm}

\begin{proof} 
In this proof, we denote by $F$ an irreducible divisors of $X'$ such that $G_F:=f(F)$ is a component of $D_Z$. 
Then, up to some $f$-exceptional divisors (hence $g$-exceptional), we have 
$$
(D')^{\verti}=\sum_{F \subset X'}(1-\frac{1}{m_F}) \cdot F,
$$ 
where $m_F$ is the multiplicity of $F$ in $D$. 
On the other hand, when we write the multiplicity of $G_F$ in $D_Z$ as $m_{G_F}$,  
we have
$$
f^{\star}(D_Z)=\sum_{F \subset X'}  t_F \cdot (1-\frac{1}{m_{G_F}})F. 
$$
Then, by the definition, we have 
$$
t_F \cdot  (1-\frac{1}{m_{G_F}})=(t_F-1)+(1-\frac{1}{m_F})-(\frac{t_F}{m_{G_F}}-\frac{1}{m_F}).
$$
Further, by the definition, we have $m_{G_F}\leq t_F \cdot m_F$, and thus we have $(\frac{t_F}{m_{G_F}}-\frac{1}{m_F})\geq 0$. 
If we define the effective divisor  $\Delta$ by 
$$
\Delta:=\sum_{F \subset X'}(\frac{t_F}{m_{G_F}}-\frac{1}{m_F})\cdot F, 
$$
then the above equality  gives the conclusion $f^{\star}(D_Z)=D(f)_{/D_Z}+(D')^{\verti}-\Delta$. 
\end{proof}

\begin{rem} The proof applies more generally when the pair $(X,D)$ to be lc, not necessarily smooth. The proof does not require the notion of sRC or of sRC quotient, only the fact that $(K_Z+D_Z)$ is pseudo-effective (although without these notions, we do not have any geometric description of the situations to which the result applies).
\end{rem}

\begin{cor} We keep the above notations. 
Then we have\,$:$
\begin{itemize}
\item If $B\subset Z$ is an irreducible divisor not contained in $D_Z$, any non reduced component of $f^{\star}(B)$ is $g$-exceptional. 
\item If $G$ is a component of $D_Z$ and $F$ is an irreducible component of $f^{\star}(G)$ that is not $g$-exceptional, then $t_F \cdot m_F=m_G$.
\end{itemize}
More precisely, we obtain
$$
f^{\star}(G)=E+A+m_G \cdot B+C
$$
where $E$ is the $g$-exceptional part of $f^{\star}(G)$, 
$A$ is reduced and contained in $D$ with all of its components of $D$-multiplicity $m_G$, 
while $B$ has no component in $D$, and $C$ consists of non-reduced components of $f^{\star}(G)$ contained in $D$.
\end{cor}

\begin{proof} Note that $A$ is not empty thanks to \cite{GHS03}. 
The first assertion follows directly from the preceding theorem. The second follows from its proof and the fact that $C'\cdot \Delta=0$. 
The last one from the fact that if $F$ is not $g$-exceptional, we have $t_F \cdot m_F=m_G>1$. Thus if $F$ is not in $D$, we have $m_F=1$, and thus $t_F=m_G$. If $F$ is reduced, $t_F=1$, and thus $m_F=m_G$. 
\end{proof}

\begin{rem} If $G$ is a component of $D_Z$, and if $F$ is an irreducible component of $f^{\star}(G)$ and not $f$-exceptional, then $t_F \cdot m_F\geq m_G>1$. Hence if $m_F=1$ (that is, $F$ not contained in $D'$), 
then $t_F\geq m_G>1$, and this may occur for $F$ not $g$-exceptional. 
Otherwise, the divisor $F$ is a component of $(D')^{\verti}$.
\end{rem}

\end{document}